\documentclass[9pt]{jdg-p-1-2}
\usepackage{amsmath}
\newtheorem{theorem}{Theorem}[section]
\newtheorem{lemma}[theorem]{Lemma}
\newtheorem{corollary}[theorem]{Corollary}
\newtheorem{proposition}[theorem]{Proposition}
\newtheorem*{theorem*}{Theorem}
\newtheorem*{acknowledgement*}{Acknowledgement}

\theoremstyle{definition}
\newtheorem{definition}[theorem]{Definition}
\newtheorem*{claim*}{Claim}
\newtheorem{example}[theorem]{Example}

\theoremstyle{remark}
\newtheorem{remark}[theorem]{Remark}

\newcommand{\R}[1]{\mathbb{R}^{#1}}
\newcommand{\gt}[0]{\tilde{g}}

\newcommand{\Rm}[0]{\operatorname{Rm}}

\newcommand{\Rc}[0]{\operatorname{Rc}}

\newcommand{\defn}[0]{\doteqdot}
\newcommand{\ve}[1]{{\mathbf{#1}}}
\newcommand{\pd}[2]{\frac{\partial #1}{\partial #2}}
\newcommand{\pdt}[0]{\frac{\partial}{\partial t}}
\newcommand{\gh}[0]{\hat{g}}

\newcommand{\Hol}[0]{\operatorname{Hol}}
\newcommand{\hol}[0]{\operatorname{\mathfrak{hol}}}
\newcommand{\mf}[1]{\mathfrak{#1}}

\newcommand{\End}[0]{\operatorname{End}}
\newcommand{\SO}[0]{SO}
\newcommand{\Id}[0]{\operatorname{Id}}
\newcommand{\dm}[0]{\operatorname{dim}}
\newcommand{\wh}[1]{\widehat{#1}}

\newcommand{\Ph}[0]{\wh{P}}
\newcommand{\Pb}[0]{\bar{P}}
\newcommand{\Rh}[0]{\wh{R}}
\newcommand{\Rb}[0]{\bar{R}}
\newcommand{\Th}[0]{\wh{T}}
\newcommand{\Tb}[0]{\bar{T}}

\newcommand{\mc}[1]{\mathcal{#1}}
\newcommand{\Kc}[0]{\mc{K}}
\newcommand{\Hc}[0]{\mc{H}}
\newcommand{\Qc}[0]{\mc{Q}}
\newcommand{\Uc}[0]{\mc{U}}
\newcommand{\Sc}[0]{\mc{S}}

\newcommand{\Wc}[0]{\mc{W}}
\newcommand{\Vc}[0]{\mc{V}}
\newcommand{\WM}[0]{\wedge^2T^*M}
\newcommand{\WMp}[0]{\wedge^2T_p^*M}

\newcommand{\im}[1]{\operatorname{image}(#1)}

\newcommand{\tsul}[3]{{#1}^{#2}_{#3}}

\newcommand{\sg}[0]{\sigma}
\newcommand{\FM}[0]{F(M)}
\newcommand{\FMt}[0]{\widetilde{F(M)}}
\newcommand{\OM}[0]{\mathcal{O}(M)}
\newcommand{\OMt}[0]{\widetilde{\mathcal{O}(M)}}
\newcommand{\GL}[0]{GL(n , \mathbb{R})}
\newcommand{\gl}[0]{\mathfrak{gl}(n, \mathbb{R})}
\newcommand{\La}[0]{\Lambda}

\numberwithin{equation}{section}

\begin{document}

\title[Ricci flow and the holonomy group]{Ricci flow and the holonomy group}

\author{Brett L. Kotschwar}
\address{Max Planck Institute for Gravitational Physics, Am M\"{u}hlenberg 1, D-14476 Golm, Germany}
\email{brett.kotschwar@aei.mpg.de}

\thanks{The author was supported in part by NSF grant DMS-034540}

\date{December 2010}

\begin{abstract}
 We prove that the restricted holonomy group of a complete smooth
solution to the Ricci flow of uniformly bounded curvature cannot
spontaneously contract in finite time; it follows, then, from an
earlier result of Hamilton
that the holonomy group is exactly preserved by the equation. 
In particular, a solution to the Ricci flow may be K\"{a}hler or
locally reducible (as a product) at $t= T$ if and only if the same is true of 
$g(t)$ at times $t\leq T$.
\end{abstract}
\maketitle

\section{Introduction}

We consider solutions to the Ricci flow
\begin{equation}\label{eq:rf}
  \pdt g = -2\Rc(g),
\end{equation} 
an evolution equation for a smooth family of Riemannian metrics $(M^n, g(t))$.  
A well-known consequence of Hamilton's strong maximum principle for systems \cite{Hamilton4D} is the following
characterization of the image of the curvature operator $\Rm:\WM\to\WM$ of a solution to \eqref{eq:rf}
when this operator is positive semidefinite. 
\begin{theorem*}[Hamilton]
Suppose $g(t)$ is a solution to \eqref{eq:rf} on $M\times [0, T]$ satisfying $\Rm(g(t))\geq 0$.
Then there exists $\delta > 0$ such that, for $t\in (0, \delta)$, $\im{\Rm(g(t))} \subset \WM$ is a smooth subbundle
invariant under parallel translation with respect to $g(t)$ and closed under the bracket
\begin{equation}\label{eq:bracketdef}
 [\omega, \eta]_{ij} = g^{kl}(\omega_{ik}\eta_{lj} - \omega_{jk}\eta_{li}).
\end{equation}
Moreover, for any $0 < t_1 < t_2 \leq T$, $\im{\Rm(g(t_1))} \subset \im{\Rm(g(t_2))}$.
\end{theorem*}
The theorem is of particular utility in low dimensions, where there are few possibilities for the
subalgebra $\Rm(\wedge^2T_p^*M)\subset \mathfrak{so}(n)$.  In three dimensions, for example, it implies that such a solution
 must have $\Rm(g(t)) > 0$ for $t > 0$ or split locally as a metric product. 
The strict code for membership in the class of of solutions with nonnegative curvature operator
may lead one to wonder
what possibilities there are for a solution $g(t)$
that attains $\Rm(g(t_0)) \geq 0$ everywhere only after some elapsed time $t_0 > 0$.
The condition $\Rm(g(t))\geq 0$ will be preserved for $t > t_0$, and a solution that splits locally 
for $t_0< t < T$ must likewise split at $t= t_0$, but we have no information on the properties of the solution
prior to $t_0$. In particular, we cannot dismiss the possibility 
that such a solution could split spontaneously at $t_0$.  One may wonder, more generally,	
whether it is possible for any solution (on a manifold with compatible topology)
 to acquire a novel local metric splitting within finite time.
Here, one's intuitive picture of the Ricci flow as a ``heat equation'' for Riemannian metrics
seems at odds with such a phenomenon. Surely it
must violate some principle of unique continuation. The basic question this paper seeks to answer is: which one?
  
Our main result is the following theorem. Here $\Hol^0(g(t))$ denotes the reduced holonomy group of $g(t)$.
\begin{theorem}\label{thm:holonomy}
Suppose $g(t)$ is a smooth complete solution to \eqref{eq:rf} on $M\times[0, T]$ of uniformly bounded curvature.
Then $\Hol^0(g(t)) \subset \Hol^0(g(T))$ for all $0 \leq t \leq T$.
\end{theorem}

Theorem \ref{thm:holonomy} is the ``backwards-time'' analog of the observation of Hamilton
(cf. \cite{Hamilton4D}, \cite{HamiltonSingularities}) that the holonomy group of a smooth solution
to the Ricci flow cannot expand within its lifetime. Thus one actually has $\Hol^0(g(t)) = \Hol^0(g(0))$ along the flow.
One consequence is an affirmation of the expectation above 
that locally product metrics are, in a sense, rigid within the class of solutions to Ricci flow.
\begin{corollary}\label{cor:holonomy}
Let $(M, g(t))$ be as in Theorem \ref{thm:holonomy}. Then $(M, g(T))$ is locally reducible (respectively, K\"ahler) if and only
if $(M, g(t))$ is locally reducible (K\"ahler) for $0 \leq	t < T$. 
\end{corollary}

One can equivalently phrase Theorem 1.1 in terms of the time-invariance
of the dimensions of the spaces of $\nabla_{g(t)}$-parallel tensors.
\begin{theorem}\label{thm:paralleldim}
If $(M, g(t))$ is as in Theorem \ref{thm:holonomy}, and $\eta\in C^{\infty}(T^k_l(M))$
satisfies $\nabla_{g(T)}\eta = 0$, then there exists a smooth family $\eta(t)\in C^{\infty}(T^k_l(M))$
for $t\in [0, T]$ such that $\nabla_{g(t)} \eta(t) = 0$ and $\eta(T) = \eta$. 
\end{theorem}

Since the reduced holonomy groups $\Hol^0_p(g(t))$ are connected Lie subgroups of $\SO(T_pM) \cong \SO(n)$,
Theorem \ref{thm:holonomy} is equivalent to the following infinitesimal reformulation
(with the choice $\mathcal{H} = \hol(g(T))$).

\begin{theorem}\label{thm:rangerm}	
Let $g(t)$ be a complete solution to \eqref{eq:rf} on $M^n\times[0, T]$ with 
$\sup|\Rm(x,t)| \leq K_0$.
Suppose there exists a smooth subbundle $\mathcal{H}\subset \WM$ that is 
invariant by $\nabla_{g(T)}$-parallel translation and closed under the bracket 
$[\cdot, \cdot]_{g(T)}$.  Then, if $\im{\Rm(g(T))} \subseteq \mathcal{H}$, 
it follows that $\im{\Rm(g(t))} \subseteq\mathcal{H}$ and
that $\mathcal{H}$ remains invariant by $\nabla_{g(t)}$-parallel translation
and closed under the bracket $[\cdot, \cdot]_{g(t)}$ for all $t\in [0, T]$.  Moreover, 
$\hol_{p}(g(t))\subseteq \mathcal{H}_p$ for all $(p, t)\in M\times[0, T]$.
\end{theorem}

We divide the proof of Theorem \ref{thm:rangerm} into several steps.
In Section \ref{sec:ncasuc}, we reduce it to a problem of unique continuation for a certain system; this is
Theorem \ref{thm:rangermredux}.  In Section \ref{sec:pdeode} 
we embed this system in a larger (closed) system of coupled partial- and ordinary-differential inequalities.
The bulk of the work is the verification that this larger system is indeed closed; for this we must perform
a rather careful analysis of the evolution equations of the components of our system. For the unique continuation,
we ultimately appeal to a special case of an earlier result of the author \cite{Kotschwar} 
for parabolic PDE-ODE systems. The approach in that reference was inspired by work of Alexakis \cite{Alexakis} 
on weakly-hyperbolic systems arising in the study of the 
vacuum Einstein equations.   

We remark that, although we restrict our attention to the Ricci flow in the present paper, the basic method
can be applied to study the holonomy of families of metrics arising from other geometric evolution equations.  
For example, a result analogous to Theorem \ref{thm:holonomy} holds for the metrics induced by the mean curvature flow of
hypersurfaces in Euclidean space (and, with additional conditions, in more general ambient spaces); 
we intend to address this in a future note.

\section{Motivation: non-expansion of holonomy.}

As we mentioned above, it is a result of Hamilton (cf. \cite{Hamilton4D}, \cite{HamiltonSingularities})
that a solution to Ricci flow with holonomy initially restricted to some subgroup of $\SO(n)$ will continue to have its holonomy so restricted.
For this paper, the statement of this ``non-expansion'' result we have in mind is the following.
\begin{theorem}[Hamilton]\label{thm:holonomynonex}
 Suppose $g(t)$ is a smooth complete solution to \eqref{eq:rf} with $g(0) = g_0$ and 
  $|\Rm(g(x, t))|\leq K_0$ on $M^n\times [0, T]$. If $\Hol^0(g_0) = G \subset \SO(n)$, 
we have $\Hol^0(g(t))\subset G$ for $0\leq t \leq T$.
\end{theorem}

Theorems \ref{thm:holonomy} and \ref{thm:holonomynonex} are statements about the backwards- and forwards-time
behavior of a solution to a (weakly-) parabolic system, and, despite their apparent symmetry, 
require rather different methods of proof. For the purpose of comparison, we will discuss two proofs of Theorem \ref{thm:holonomynonex} in detail. 
The first is an elementary combination of Berger's classification \cite{Berger}, de Rham's splitting theorem \cite{DeRham}, 
and the uniqueness of solutions for the Ricci flow \cite{Hamilton3D}, \cite{ChenZhu}. 
The second, which we defer to the appendix, is essentially self-contained and
closer to the argument suggested in \cite{HamiltonSingularities}.

We first give an example to show that, in general, one cannot  dispense with the restriction that $g(t)$ be complete
(cf. also the similar example on p. 247 of \cite{ChowLuNi}).
 
\begin{example}[Flat-sided sphere]\label{ex:holonomyjump}
Let $U\subset S^2$ be a proper open set, $x_0\in S^2\setminus U$, and $h_0$ a metric on $S^2$ of Gaussian curvature $K_{h_0}\geq 0$ satisfying
$K_{h_0}\equiv 0$ on $U$ but $K_{h_0}(x_0) > 0$. One can take, e.g., $x_0$ to be the north pole, $U$ a small disk about the south pole
and $\phi \in C^{\infty}(S^2, [0, 1])$
with $\phi \equiv 1$ on the upper hemisphere and $\phi \equiv 0$ on $U$. By the theorem of Kazdan-Warner
\cite{KazdanWarner}, one can find a metric $h_0$ with $K_{h_0} = \phi$, and,
for this metric, there exists $T > 0$ and a solution $h(t)$ to the Ricci flow defined for $t \in [0, T)$ with $h(0) = h_0$.
For any $a > 0$, we can define a solution $g_a(t)$ to Ricci flow on $U$ by  
\[
  g_a(x, t) \defn \left\{\begin{array}{ll} \left.h_0\right|_U(x) & (x, t) \in U \times [0, a]\\
					   h(x, t - a) & (x, t) \in U\times (a, a + T).   

    \end{array}\right.
\]
For $0 < t\leq a$, $K_{g_a(t)}\equiv 0$, but the strong maximum principle implies $K_{h(t)} > 0$ for $t >0 $, so $K_{g_a(t)} > 0$ for $t > a$.
Thus $(U, g_a(t))$ satisfies $\Hol^0(g_a(t)) = \{Id\}$ for $t\leq a$, but $\Hol^0(g_a(t)) = SO(2)$
for $a < t < T$.
\end{example}

\subsection{Non-expansion via Berger's Classification}

All of the ingredients of the proof below can be found, for example, in the combination of the references \cite{HamiltonSingularities} and \cite{Joyce}. 
The argument can be summarized very succinctly. In the category of complete solutions to the Ricci flow with
bounded curvature, any initial isometries are preserved, and product, K\"{a}hler, and Einstein initial
data extend uniquely to solutions of the same type. With the splitting theorem \cite{DeRham}
and the classification theorem \cite{Berger}
as it is now understood, this is enough to conclude that any restriction of the initial holonomy
is shared by the solution at later times.

We will refer to the following modern version of Berger's theorem (cf., e.g., Theorem 3.4.1, \cite{Joyce}).
\begin{theorem}[Berger]\label{thm:berger}
If $M^n$ is simply connected and $g$ is irreducible, then either $g$ is symmetric or exactly one of the
following hold:
\begin{enumerate}
 \item $\Hol^0{(g)} = SO(n)$,
 \item $n=2m$ with $m\geq 2$, and $\Hol^0(g) = U(m)$ in $SO(2m)$,
 \item $n=2m$ with $m\geq 2$, and $\Hol^0(g) = SU(m)$ in $SO(2m)$,
 \item $n=4m$ with $m\geq 2$, and $\Hol^0(g) = \operatorname{Sp}(m)$ in $SO(4m)$,
 \item $n=4m$ with $m\geq 2$, and $\Hol^0(g) = \operatorname{Sp}(m)\cdot \operatorname{Sp}(1)$ in $SO(4m)$,
 \item $n= 7$ and $\Hol^0(g) = G_2$ in $SO(7)$, or
 \item $n= 8$ and $\Hol^0(g) = \operatorname{Spin}(7)$ in $SO(8)$.
\end{enumerate}
\end{theorem}

\begin{proof}[First proof of Theorem \ref{thm:holonomynonex}]

First, we may assume that $M$ is simply connected, as $\Hol^0(\tilde{g_0}) = \Hol^0(g_0)$ 
if $\tilde{g_0}$ is the lift of $g_0$ to the universal cover of $M$. 
  We may also assume $\Hol^0(g_0)$ is irreducible.
Otherwise, by  de Rham's splitting theorem, $(M, g_0)$ splits as a global product 
\[
(M, g_0) \cong (N_1\times N_2\times \cdots\times N_m, g_1 \oplus g_2\cdots \oplus g_m).
\]
Each metric $g_i$ will be complete and of bounded curvature $|\Rm(g_i)|\leq K_0$, and so, by the existence
theorems of Hamilton \cite{Hamilton3D} and Shi \cite{Shi}, each
factor $N_i$ will admit a complete solution $g_i(t)$ of bounded curvature with $g_i(0) = g_i$ on some small time interval 
$[0, T_i]$ (with $T_i$ depending only $K_0$ and $\operatorname{dim}(N_i)$). Then $\gh(t) \defn g_1(t)\oplus g_2(t)\cdots \oplus g_m(t)$ will be a complete solution of bounded curvature 
on $M \times [0, \delta]$ for $\delta >0$ equal to the minimum of the $T_i$.  But, by uniqueness, there is only
one solution of bounded curvature with initial data $g_0$, hence $g(t) \equiv \gh(t)$ on $M\times [0, \delta]$.
The argument may then be iterated on intervals of uniform size
to obtain the agreement of $g(t)$ with a product solution on all of $M\times [0, T]$.
Since we may then consider each factor independently, we may as well assume that $g_0$ is irreducible.

Now we consider each case of Theorem \ref{thm:berger} in turn. 
Suppose first that $g_0$ is symmetric. The uniqueness of solutions and the 
diffeomorphism invariance of the equation imply that $\operatorname{Isom}(g(0)) \subset \operatorname{Isom}(g(t))$. 
For a general metric $g$, denote by $A(g)$ the set of isometries
$A(g) \defn \left.\{\;\sigma_q\in \operatorname{Iso}_q(g)\;\right|\; \sigma_q^2 = Id\;\}$.
Since the composition law of $\operatorname{Isom}(g(t))\subset \operatorname{Diff}(M)$ and the set of any isometry's fixed points 
are independent of the metric, the preservation of initial isometries also implies $A(g_0)\subset A(g(t))$.  
In particular, $g(t)$ remains symmetric for $t > 0$. But for a symmetric metric $g$, each fixed representative $\Hol_p^0(g)$ 
of the isomorphism class of $\Hol^0(g)$ can be described explicitly as the subgroup of squares
of involutive isometries fixing $p$ (cf. Proposition 3.35 of \cite{Joyce}). Symbolically,  
\[
  \Hol_p^0(g) = J_p(g) \defn 
  \operatorname{Iso}_p(g)\cap \left.\{\;\sigma_q \circ \sigma_r\; \right|\; \sigma_q, \sigma_r \in A(g)\;\}. 
\]
Then $A(g_0)\subset A(g(t))$ implies $J_p(g_0)\subset J_p(g(t))$ and, since $g(t)$ is symmetric, that
\[
  \Hol^0_p(g(t)) = J_p(g(t)) \subset J_p(g_0) = \Hol_p(g_0).
\]

Therefore we are left with the seven alternatives on Berger's list.
The first of these is uninteresting, of course, as $\Hol^0(g) \subset \SO(n)$ for any metric $g$. 
The second, $\Hol^0(g_0) = U(n/2)$, implies $g_0$ is K\"{a}hler, and it is well-known that 
from a K\"{a}hler initial metric of bounded curvature
one can construct a K\"{a}hler solution of bounded curvature
by the solution of an appropriate parabolic Monge-Ampere equation for the potential. 
This solution may, a priori, only exist for a short time,
but for this period
we must have $g(t) \equiv \gh(t)$ by uniqueness (and hence $\Hol^0(g(t))\subset U(n/2)$). We may then
iterate as before to conclude the same on the entire interval of existence for $g(t)$.

This leaves five cases.  However, in each of these, $g_0$ 
is necessarily Einstein (cf. pp. 53-55 of \cite{Joyce}). (In fact, in the cases $SU(m)$, $\operatorname{Sp}(m)$, $\operatorname{Spin}(7)$, or $G_2$, 
the metric must be Ricci-flat.) But, associated to Einstein initial data $\Rc(g_0) = \rho g_0$, one can construct
the Einstein solution  $\gh(t) = (1 - 2\rho t) g_0$ which moves only by homothetical scaling.  The holonomy is obviously unchanged for this
solution and it is unique among (at least) those of uniformly bounded curvature. Thus $\gh(t) = g(t)$
and $G = \Hol^0(\gh(t)) = \Hol^0(g(t))$.
\end{proof}

\subsection{Berger's theorem and non-contraction of holonomy}

It is natural to ask whether one can fashion an analogous argument along for Theorem \ref{thm:holonomy}.
The answer seems to be ``only partially.'' The failure of this argument to extend to all cases
was, in fact, the starting point for the work in the present paper.

Of the three primary components of the preceding proof, we nevertheless retain at least two. The classification component,
coming from Berger's and de Rham's theorems and their consequences, is as applicable to $g(T)$ as it was to $g(0)$.
From \cite{Kotschwar}, we also have a counterpart to the uniqueness component: two complete solutions $g(t)$, $\gt(t)$ to \eqref{eq:rf} of uniformly bounded curvature 
that agree at $t = T > 0$ must agree at times $t < T$. From this, it follows that any isometries of $g(T)$
are shared by $g(t)$ for $t < T$, and that $g(T)$ is Einstein only if $g(t)$ is as well for $t < T$.

What we lack, rather, is the ability to construct by hand the special ``competitor'' solutions to
extend the data $g(T)$ to a solution of the same type for times $t < T$. Of course, if $g(T)$ is Einstein, 
we may still construct an extension by homothetical scaling of $g(T)$. However, when $g(T)$ is K\"{a}hler,
we cannot simply construct a K\"ahler extension $\gt(t)$ for $T - \delta < t \leq T$ by the method above, since
we must now specify instead the data for the potential at time $T$. Such 
``terminal-value'' parabolic problems are \emph{ill-posed} and lack solutions in general.  The analogous terminal-value problems for the Ricci (or Ricci - De Turck) flows
are also ill-posed, and this is an impediment, in particular, to the construction of a product extension $\gt(t) = \gt_1(t)\oplus \gt_2(t)$
for $t < T$ from product data $g(T) = g_1\oplus g_2$ on $N_1\times N_2$.  The trouble is that, 
while the product metric $g_1\oplus g_2$ belongs to $RF(M, T)$ -- the 
``image'' of the time-$T$ Ricci flow operator on $M$, we do not know whether either of the factors
$g_i$ belong to $RF(N_i, \delta)$ for \emph{any} $\delta > 0$.

While Theorem \ref{thm:holonomy} is not simply reducible to the backwards-uniqueness of solutions 
to \eqref{eq:rf}, we will show, nevertheless, that it is equivalent to the backwards-uniqueness
of a certain larger, mixed parabolic and ordinary-differential, system. 
The argument we will describe in the next section 
(and carry out in those following) will be essentially self-contained and, in particular,
independent of the theorems of Berger and de Rham.    

\section{Non-contraction of $\Hol^0(g(t))$ as a problem of unique continuation.}
\label{sec:ncasuc}

Our basic strategy is to interpret restricted holonomy as a condition on the operator
$\Rm:\WM\to \WM$. (This is also the basis of Hamilton's approach to non-expansion of holonomy 
in \cite{HamiltonSingularities}). This characterization is natural since the curvature effectively determines 
the holonomy Lie algebra
(in a manner we will review below), but it offers an additional advantage for our purposes
in that the curvature operator, unlike the metric, satisfies a strictly parabolic equation.

The representation of the holonomy Lie algebra $\hol(g(T))$ on $TM$
gives a subbundle of $\WM$ that is invariant under parallel translation and closed under the Lie-bracket
given by \eqref{eq:bracketdef}. 
The image of the curvature operator is
contained in $\hol(g(T))$ and, as $\Rm(g(T))$ is self-adjoint, its kernel at each $p$ therefore contains $\hol_p(g(T))^{\perp}$. 
The bundle $\hol(g(T))^{\perp}$
is likewise closed under parallel translation, though not in general under the Lie bracket. The following 
observation shows that (as in Theorem \ref{thm:rangerm}) we may as well consider
any parallel subalgebra $\mathcal{H}$ containing $\Rm(g(T))$, $\hol(g(T))$
being, in a sense, the minimal such $\mathcal{H}$. 

\begin{lemma}\label{lem:ambrosesinger}
Suppose $\Hc\subset \wedge^2(T^*M)$ is a smooth distribution closed under parallel transport and the Lie bracket 
\eqref{eq:bracketdef}.
If, for all $p\in M$, $\im{\Rm(g(p))} \subset \Hc_p$, then $\mf{hol}_p(g) \subset \Hc_p$.
\end{lemma}
\begin{proof}
This follows easily from the Ambrose-Singer theorem \cite{AmbroseSinger} (cf. also Besse \cite{Besse}, Theorem 10.58)
which says that the elements of the leftmost union in the chain of inclusions
\[
 \bigcup_{q\in M, \gamma \in \Omega_{p, q}, \omega \in \wedge^2T^*_pM} 
(\tau_{\gamma} \circ \Rm(q)\circ\tau_{\gamma}^{-1})(\omega) 
\subset \bigcup_{q\in M, \gamma \in \Omega_{p, q}}\tau_{\gamma}^*(\Hc_q) \subset \Hc_p.
\]
generate $\hol_p(g)$. 
Here $\Omega_{p,q}$ represents the space of piecewise smooth paths $\gamma:[0, 1] \to M$ with $\gamma(0) = p$,
$\gamma(1) = q$ and $\tau_{\gamma}$ represents the extension of parallel transport along the path $\gamma$ to two-forms.
\end{proof}

Assuming then we have such a $\mathcal{H}\subset \WM$, we consider its perpendicular complement $\mathcal{K}\defn \Hc^{\perp}$
and associated orthogonal projection operator $\Ph_T: \WM\to\mathcal{K}$.  
Although we ultimately wish to show that $\left.\Rm(g(t))\right|_{\Kc} \equiv 0$, we do not know
a priori whether, for $t < T$, the fibers of $\Hc$ and $\Kc$ are complementary orthogonal subspaces (or that those of $\Hc$ are closed under
the bracket \eqref{eq:bracketdef}) relative to $g(t)$.  Thus we first define time-dependent extensions
$H(t)$ and $K(t)$ for $\Hc$ and $\Kc$ that retain these properties on $[0, T]$. Then we prove $\left.\Rm(g(t))\right|_{K(t)} \equiv 0$ 
(hence $\im{\Rm(g(t))}\subset H(t)$) and use this to show $H(t)\equiv \Hc$ and $K(t) \equiv \Kc$.

We define $H(t)$ and $K(t)$ as the images of the families of projection maps $\Pb(t)$ and $\Ph(t)$ extending $\Pb_T$ and $\Ph_T$. We have 
$\nabla_{g(T)} \Pb_T \equiv \nabla_{g(T)} \Ph_T \equiv 0$ and $\Rm(g(T))\circ\Ph_T \equiv 0$, and, by spelling out
the mandate that they remain complementary orthogonal projections,
it is not hard to determine
what these extensions $\Pb(t)$ and $\Ph(t)$ ought to be, namely,
the solutions to $D_t \Ph = 0$ on $[0, T]$ with $\Pb(T) = \Pb_T$ and $\Ph(T) =\Ph_T$. 
Here $D_t$ represents a time-like vector tangent
to the submanifold of $g(t)$-orthonormal frames in the product of the frame bundle with the interval:
$F(M)\times [0, T]$ (see \eqref{eq:projode} and Section \ref{ssec:omt}).  This extension, in any event,
is achieved by solving an ODE on each fiber of $\WMp$. With $\Ph(t)$ so obtained, we arrive at the following
``backwards-uniqueness'' problem: to show
 $\Rm\circ\Ph(t) \equiv 0$ and $\nabla\Ph(t) \equiv 0$ for all $0\leq t < T$, given their vanishing at $t = T$.
Once it has been established, all that remains is to verify that $K(t) = \im{\Ph(t)}$ is in fact constant in time.  This is a consequence
of the equation satisfied by $\Ph(t)$, and we do this in Lemma \ref{lem:imagerm} below.

The remainder of the present section will be dedicated to the reduction of Theorem \ref{thm:rangerm}
to a precise statement of the backwards-uniqueness problem described above; 
this will be Theorem \ref{thm:rangermredux}.

\subsection{Some preliminaries.}
The following elementary observation
will in fact be essential to the computations in Section \ref{sec:pdeode}.
\begin{lemma}\label{lem:la}
 Suppose $V$ is a vector space with an inner product $\langle\cdot, \cdot\rangle$ and a consistent 
Lie bracket $[\cdot, \cdot]$. 
If $H\subset V$ is a subalgebra, and $K \defn H^{\perp}$, then $[H, K] \subset K$.
\end{lemma}
\begin{proof}
  The assumption of consistency implies that the trilinear map
\[
  (X, Y, Z) \mapsto \left\langle [X, Y], Z \right\rangle
\]
is fully antisymmetric. Thus, if $h_1$, $h_2\in H$ and $k\in K$, we have
\[
  \left\langle [h_1, k], h_2\right\rangle = - \left\langle [h_1, h_2], k\right\rangle = 0
\]
as $[h_1, h_2] \in H = K^{\perp}$.
\end{proof}

Related to the trilinear form in the above proof is the following operator, which we will need
to identify in certain of our computations that follow.
\begin{definition}\label{def:tdef}
  Suppose $V$ is a vector space with inner product $\langle\cdot , \cdot \rangle$  and Lie bracket $[\cdot , \cdot ]$.
  Let
\[
    \mathcal{T}: \End(V)\times\End(V)\times \End(V) \to V^*\otimes V^*\otimes V^*
\]  
be the operator defined by
\[
  \mathcal{T}[A, B, C](v_1, v_2, v_3) 
      \defn \left\langle \left[ A(v_1), B(v_2)\right], C(v_3)\right\rangle
\]
for $A$, $B$, $C\in\End(V)$, $v_i\in V$.  
\end{definition}

For completeness, we include the proof of a few elementary properties of projection maps and parallel translation that
we will use in the sequel.

\begin{lemma}\label{lem:vb}
 Suppose $M$ is connected and 
$\pi: V\to M$ is a smooth $m$-dimensional vector bundle with connection $D$ and a compatible metric $h$ on its fibers.
(We will also use $D$ to represent the induced connection $D:\End(V) \to T^*M\otimes \End(V)$).
Let $H\subset V$ be an $l$-dimensional smooth subbundle, $K = H^{\perp}$, and $\Pb: V \to H$, $\Ph:V\to K$ the 
$h$-orthogonal projections onto $H$ and $K$.  
\begin{enumerate} 
 \item For any $X\in TM$, 
\[
\Pb\circ D_X\Pb\circ \Pb =
			  \Ph\circ D_X\Pb \circ \Ph = \Pb\circ D_X\Ph\circ\Pb = \Ph \circ D_X\Ph\circ \Ph = 0. 
\]
\item The following are equivalent: $H$ is closed under parallel translation, $K$ 
is closed under parallel translation, $D\Pb \equiv 0$, and $D\Ph \equiv 0$. 
\end{enumerate}
\end{lemma}
\begin{proof}

For the first claim, we fix $X\in TM$, and differentiate both sides of the identity $\Pb\circ\Pb =\Pb$ to obtain
\begin{align*}
    D_X\Pb\circ\Pb +  \Pb\circ D_X\Pb &= D_X\Pb.
\end{align*}
Pre- and post-composing both sides of this result with $\Pb$ and using again the above identity, we arrive at
\[
    2 \Pb\circ D_X\Pb \circ \Pb = \Pb\circ D_X\Pb\circ \Pb,
\]
from which we conclude $\Pb\circ D_X\Pb\circ \Pb = 0$.  For the second equality in (1), we differentiate
both sides of $\Ph\circ\Pb = 0$ to obtain
\[
 D_X \Ph \circ \Pb + \Ph \circ D_X\Pb = 0.
\]
If we now pre- and post-compose both sides with $\Ph$, the first term on the left vanishes, and we are left with
$\Ph \circ D_X\Pb\circ \Ph = 0$.  The identities for $D_X \Ph$ follow similarly.

For the second claim, first note that, since $\Pb + \Ph = \Id: V\to V$, we have $D\Pb = 0$ if and only if $D\Ph = 0$.
Suppose now that $D\Pb = D\Ph = 0$.  Given $p$, $q\in M$, $X\in H_p$ and $\gamma:[0, 1] \to M$ 
a smooth curve joining $p$ to $q$, define $X(t) \in V_{\gamma(t)}$ by parallel transport along $\gamma$.
If $T = \gamma_*(\frac{d}{dt})$, then $D_T\Ph = 0$ and $D_{T} X = 0$ along $\gamma$.
But the compatibility of the metric with $D$ implies $f(t) = |\Ph(X(t))|_h^2$ satisfies 
\[
f^{\prime}(t) = 2 \left\langle D_T\Ph(X) + \Ph\left(D_T(X)(t)\right), \Ph(X(t))\right\rangle = 0 
\]
and $f(0) = 0$.  Thus $f\equiv 0$ and, in particular, $\Ph(X(1)) = 0$, i.e., $X(1)\in K_q^{\perp} = H_{q}$.
So $H$ is closed under parallel translation. Similarly, $K$ is closed under parallel translation.

Suppose then that, on the other hand, $H$ is invariant under parallel translation.
  Let $p$, $q \in M$, $\gamma:[0, 1]\to M$ a smooth path connecting $p$ and $q$, 
  $\{V_i\}_{i=1}^l$ and $\{V_i\}_{i=l+1}^m$ be orthogonal
orthogonal bases for $H_p$ and $K_p$ respectively and $V_i(t)$ the parallel transports of $V_i$ along $\gamma$. Then
$V_i(t)\in H_{\gamma(t)}$ for $t\in [0, 1]$. For any $i$, $j$, define $A_{ij}(t) = h(\gamma(t))(V_i(t), V_j(t))$.  
Then $A^{\prime}_{ij} = 0$ as above. Since $A_{ij}(0) = \delta_{ij}$, we have $A_{ij}(t) = \delta_{ij}$.
In particular, $H_q = \operatorname{span}\{V_i(1)\}_{i=1}^l = \left(\operatorname{span}\{V_i(1)\}_{i=l+1}^m\right)^{\perp}$,
 so $K_q = \operatorname{span}\{V_i(1)\}_{i=l+1}^m$.
Since $p$ and $q$ were arbitrary, $K$ is also invariant under parallel translation; obviously we can also reverse the roles of $H$
and $K$.

Finally, suppose again that $H$ (hence, now, also $K$) is invariant under parallel translation.  We wish to show that $D\Pb \equiv 0$ (which is equivalent
to $D\Ph \equiv 0$ as remarked above),
and for this it suffices to show that $h(p)(D_X\Pb(U), W) = 0$ 
for an arbitrary $p \in M$, 
$X\in T_pM$, and $U$, $W\in V_p$.  So let 
$\gamma:(-\epsilon, \epsilon)\to M$ be any smooth curve with $\gamma(0) = p$ and 
$\gamma^{\prime}(0) = X$.  Define $U(t)$ and $W(t)$ to be the parallel transports of $U$ and $W$ along $\gamma$ 
and let $k(t) = h(\gamma(t))(\Pb U(t), W(t))$.
By the first part of this lemma, we only need to check the ``off-diagonal'' components of $D_X\Pb$, that is,
the cases in which $U$ and $W$ belong to opposite summands of $V_p = H_p \oplus K_p$. So suppose first that $U\in H_p$ and $W\in K_p$.  
Since $H$ and $K$ are invariant under parallel translation, $U(t)\in H_{\gamma(t)}$ and
$W(t)\in K_{\gamma(t)}$, thus 
\[
k(t) = h(\gamma(t))(U(t), W(t)) \equiv h(p)(U, W) = 0.
\] Similarly, if $U\in K_p$ and $W\in H_p$,
then $\Pb U(t) \equiv 0$ and $k(t) \equiv 0$.  In both cases, we have $k^{\prime}(0) = h(p)(D_X U, W) = 0$.
\end{proof}

\subsection{A time-dependent family of distributions.}

Going forward, let $g(t)$ be a smooth solution of \eqref{eq:rf} on $M^n$ for $t\in [0, T]$ and define $h\defn g(T)$.  
Given a tensor field $V\in T^k_l(M)$ that is, in some sense, ``calibrated''  to the metric $g$ at $t_0 \in [0, T]$,
there is a natural (and well-known) means of extending $V$ to a family of sections $V(t)$ for $t\in [0, T]$ 
with the promise of preserving this calibration. 
Namely, one can define $V(p,t)$ in each fiber $T^k_l(T_pM)$ as the solution of the ODE
\begin{align}\label{eq:projode}
\begin{split}
	\pdt \tsul{V}{a_1a_2\cdots a_k}{b_1b_2\cdots b_l} 
	  &= R^{a_1}_c\tsul{V}{ca_2\cdots a_k}{b_1b_2\cdots b_l}
	  + R^{a_2}_c\tsul{V}{a_1c\cdots a_k}{b_1b_2\cdots b_l}
	  +\cdots + R^{a_k}_{c}\tsul{V}{a_1a_2\cdots c}{b_1b_2\cdots b_l}\\
	  &\phantom{=} - R^{c}_{b_1}\tsul{V}{a_1a_2\cdots a_k}{cb_2\cdots b_l}
	  - R^{c}_{b_2}\tsul{V}{a_1a_2\cdots a_k}{b_1c\cdots b_l}
	  -\cdots - R^{c}_{b_l}\tsul{V}{a_1a_2\cdots a_k}{b_1b_2\cdots c}
\end{split}
\end{align}
for $t\in[0, T]$ with $V(p, t_0) = V(p)$.  
Note that if $V = g(t_0)$, this procedure simply recovers the solution $g(t)$, 
and if $V$ and $W$ are related by an identification
of $TM$ with $TM^*$ according to the metric $g(t_0)$ (i.e., by raising or lowering indices), 
then $V(t)$ and $W(t)$ will be related by the analogous identification of $TM$ and $T^*M$ according to $g(t)$ on $[0, T]$.  
Likewise, a contraction of $V$ by $g(t_0)$ evolved
according to \eqref{eq:projode} will be the same contraction of $V(t)$ by $g(t)$. 
Equation \eqref{eq:projode}
is equivalent to considering the evaluation of the fixed tensor $V$ on a time-dependent local frame evolved 
so as to preserve the pairwise inner products of the elements of the frame.
We will consider a somewhat more formal variation of this identification in the next section; in the notation
presented there, the above procedure is equivalent to finding a representative $V$ satisfying $D_t V \equiv 0$.

At present, though, \eqref{eq:projode} allows us to identify the distributions $\Hc$ and $\Kc$
with convenient relatives $H(t)$ and $K(t)$.  We 
let $\Pb_T$, $\Ph_T\in \End({\WM})$ denote, respectively, the orthogonal projections onto $\Hc$ and $\Kc$
with respect to $h$, and construct $\Pb(t)$ and $\Ph(t)$ according to the procedure
\eqref{eq:projode} with $\Pb(T) = \Pb_T$ and $\Ph(T) = \Ph_T$.  Thus, in components, and here regarded as elements
of $\End({\WM})\cong T_4(M)$ (see Section \ref{ssec:notation}),
\begin{align*}
  \pdt \Pb_{abcd} &= -R_{ap}\Pb_{pbcd}-R_{bp}\Pb_{apcd}-R_{cp}\Pb_{abpd}-R_{dp}\Pb_{abcp}\\
  \pdt \Ph_{abcd} &= -R_{ap}\Ph_{pbcd}-R_{bp}\Ph_{apcd}-R_{cp}\Ph_{abpd}-R_{dp}\Ph_{abcp}.
\end{align*}
We then define 
\[
  H(t) \defn \im{\Pb(t)} \subset\WM, \quad\mbox{ and } 
\quad K(t) \defn \im{\Ph(t)} \subset \WM.
\]
We collect here the properties of these subspaces we will need in the sequel.
\begin{lemma}\label{lem:proj}
Let $g$, $h = g(T)$, $\Hc$, $\Kc$, $H(t)$, and $K(t)$ be defined as above,
and $\dim{\Hc} = k$.  For all $t\in [0, T]$, $\dm{H(t)} = k$,
$H(t)$ is closed under the Lie bracket \eqref{eq:bracketdef} with respect to $g(t)$,
and $K(t) = H(t)^{\perp}$.  Moreover, if $\mathcal{T}$ is defined as in Definition \ref{def:tdef},
we have 
\begin{equation}\label{eq:tvan}
\mathcal{T}[\Ph, \Pb, \Pb] = \mathcal{T}[\Pb, \Ph, \Pb] = \mathcal{T}[\Pb, \Pb, \Ph] = 0. 
\end{equation}
\end{lemma}
\begin{proof} The first three properties are easily verified from equations \eqref{eq:bracketdef} and
\eqref{eq:projode}. 
The last follows then from Lemma \ref{lem:la}.
\end{proof}

We now show that if it happens that $\im{\Rm(g(t))}\subset H(t)$ for all $t$,
then $H(t)$ and $K(t)$ are actually independent of time.

\begin{lemma}\label{lem:imagerm}
With $g(t)$, $H(t)$, $K(t)$ as above, suppose $\im{\Rm(g(t))} \subset H(t)$ for all $t\in [0, T]$.
Then $H(t) \equiv H(T) = \Hc$, $K(t) \equiv K(T) = \Kc$. 
\end{lemma}
\begin{proof}
Let $m = n(n-1)/2$ and $p\in M$, and $h = g(p, T)$. Choose an find an $h$-orthonormal basis 
$\{\varphi^A\}_{A = 1}^{m}$ of sections for $\WMp$ such that
$\{\varphi^A\}_{A = 1}^k$ is a basis for $\Hc_p$ and $\{\varphi^A\}_{A = k+1}^m$ a basis for $\Kc_p$.
We can then use the procedure described by equation \eqref{eq:projode} on the individual forms $\varphi^A$
to a produce a family of two-forms $\{\varphi^A(t)\}_{A=1}^m$ on $T_pM$ for $t\in [0, T]$.  
This set will be a $g(t)$-orthonormal basis for $\WMp$ for any $t$, and moreover, 
\[
  \Pb(t)\left(\varphi^A(t)\right) = \left\{\begin{array}{ll}
      \varphi^A(t) & \mbox{if}\quad A \leq k\\
      0		   & \mbox{if}\quad A > k,
    \end{array}\right. 
\]
and
\[
\Ph(t)\left(\varphi^A(t)\right) = \left\{\begin{array}{ll}
      0 & \mbox{if}\quad  A \leq k\\
      \varphi^A(t) & \mbox{if}\quad  A > k.
    \end{array}\right. 
\]
Thus $\{\varphi^A(t)\}_{A = 1}^k$ and $\{\varphi^B(t)\}_{A = k+1}^{m}$ remain bases for $H(t)$ and $K(t)$, respectively.
In fact, for $t\in [0, T]$,
\[
   \Pb(t) = \sum_{A = 1}^{k}\left(\varphi^A(t)\right)^*\otimes\varphi^A(t), \quad
   \Ph(t) = \sum_{A = k+1}^{m}\left(\varphi^A(t)\right)^*\otimes\varphi^A(t).
\]

Now, for any fixed $t$, we can choose an orthonormal basis $\{e_a\}$ of $T_pM$ relative to $g(p, t)$;
in these components $g_{ab}(p,t) = \delta_{ab}$. 
Let $M$ be the symmetric matrix defined by 
\[
R_{abcd} = - M_{AB}\varphi^A_{ab}\varphi^B_{cd}.
\]
For any $A$, at $(p, t)$ we have (observing the extended summation condition),
\begin{align*}
  \pdt \varphi^A_{ab} &= -R_{aq}\varphi^A_{qb} - R_{bq}\varphi^A_{aq}\\
    &= -R_{appq}\varphi^A_{qb}- R_{bppq}\varphi^A_{aq}\\
    &= M_{BC}\left(\varphi^B_{ap}\varphi^C_{pq}\varphi^A_{qb} + \varphi^B_{bp}\varphi^C_{pq}\varphi^A_{aq}\right)\\
\begin{split}
    &= M_{BC}\left(\left(\varphi^B_{ap}[\varphi^C, \varphi^A]_{pb} + \varphi^B_{ap}\varphi^C_{bq}\varphi^A_{qp}\right)
+ \left(\varphi^B_{bp}[\varphi^C, \varphi^A]_{ap} + \varphi^B_{bp}\varphi^C_{aq}\varphi^A_{pq}\right)\right)
\end{split}\\
\begin{split}
    &= M_{BC}\left[\left[\varphi^A, \varphi^C\right], \varphi^B\right]_{ab} - R_{apbq}\varphi^A_{qp} 
    - R_{bpaq}\varphi^A_{pq}
\end{split}\\
\begin{split}
    &= M_{BC}\left[\left[\varphi^A, \varphi^C\right], \varphi^B\right]_{ab}
  + \left(R_{apbq} + R_{pbaq}\right)\varphi^A_{qp}
\end{split}\\
\begin{split}
    &= M_{BC}\left[\left[\varphi^A, \varphi^C\right], \varphi^B\right]_{ab} - R_{bapq}\varphi^A_{qp} 
\end{split}\\
\begin{split}
    &= M_{BC}\left[\left[\varphi^A, \varphi^C\right], \varphi^B\right]_{ab} + \Rm(\varphi^A)_{ab}, 
\end{split}\\
\end{align*}
where the penultimate line follows from the Bianchi identity.  That is, 
\begin{equation}\label{eq:twoformevol}
\pdt \varphi^A = L(\varphi^A) + \Rm(\varphi^A),
\end{equation}
where  $L: \WMp\to\WMp$ is the linear map determined by
\[
  L(\varphi^A) = M_{BC}\left[[\varphi^A, \varphi^C], \varphi^B\right].
\]
Note that \eqref{eq:twoformevol} is independent of the frame $\{e_a\}$.

We claim that $L$ satisfies $L(H_p(t))\subset K_p(t)$ and $L(K_p(t))\subset H_p(t)$.
First, since $\im{\Rm(p, t)}\subset H_p(t)$, and $\Rm$ is symmetric, it follows that $K_p(t)\subset \ker{(\Rm(p, t))}$,
and hence that $M_{BC} = 0$ if $B > k$ or $C > k$.
Also, by Lemmas \ref{lem:la} and \ref{lem:proj}, we have $[H_p(t), H_p(t)]\subseteq H_p(t)$,
and $[H_p(t), K_p(t)] \subseteq K_p(t)$.  Stated in terms of the structure constants 
\[
C^{AB}_C = \left\langle[\varphi^A, \varphi^B],\varphi^C\right\rangle
\]
this is $C^{AB}_C = 0$ if $A, B \leq k$ and $C > k$, or exactly one of $A$ and $B$ is greater than
$k$ and $C\leq k$.
Now,
\begin{equation}\label{eq:lexp}
  L(\varphi^A) = \sum_{B, C \leq k}\sum_{1\leq D, E \leq m} M_{BC}C^{AC}_DC^{DB}_E\varphi^E.
\end{equation}
If $A \leq k$, then each $C^{AC}_D$ is only nonzero for $D \leq k$,  
\[
  L(\varphi^A) = \sum_{B, C, D, E \leq k} M_{BC}C^{AC}_DC^{DB}_E\varphi^E
    \defn \sum_{1\leq E\leq k} S^A_{E}\varphi^E.
\]
Likewise, if $A > k$, the only non-zero occurences of $C^{AC}_D$ in \eqref{eq:lexp} are those
with $D > k$, thus restricting the non-zero occurences of $C^{DB}_E$ in the sum to those with 
$E > k$. So
\[
  L(\varphi^A) = \sum_{B, C\leq k}\sum_{k < D, E } M_{BC}C^{AC}_DC^{DB}_E\varphi^E
    \defn \sum_{k < E\leq m} T^A_{E}\varphi^E.
\]

Thus defining
\[
  V(t) \defn \left(\varphi^1(t), \varphi^2(t),\ldots, \varphi^k(t)\right)^T,\quad
   W(t) \defn \left(\varphi^{k+1}(t), \varphi^{k+2}(t), \ldots, \varphi^m(t)\right)^T,
\]
and using that $\Rm(\varphi^A) \equiv 0$ if $A > k$, we can restate \eqref{eq:twoformevol}
as a matrix equation
\[
     \left(\begin{array}{c}\dot{V}(t)\\
	      \dot{W}(t)
          \end{array}\right) 
    = \left(\begin{array}{cc} S(t) + M(t) & 0 \\
			      0    & T(t)   \end{array}\right)
    \left(\begin{array}{c} V(t)\\
	    W(t)
          \end{array}\right).
\]
It follows, then, that for all $0\leq t \leq T$, for appropriate coefficients $E^A_{B}(t)$, 
\[
  \varphi^{A}(t) = \sum_{B=1}^k E^A_{B}(t)\varphi^{B}(T) \in H(T) = \mathcal{H}
\]
if $A \leq k$.  Similarly, if $A > k$, we have $\varphi^A(t) \in \mathcal{K}$ for all $0\leq t \leq T$. 
\end{proof}

\subsection{A restatement of Theorem \ref{thm:rangerm}}

Now we are able to frame Theorem \ref{thm:rangerm} as a problem of unique continuation.  
Under the assumptions of that theorem, we have, by
Lemma \ref{lem:proj}, a $g(t)$-orthogonal
decomposition $\WM = H(t)\oplus K(t)$ where $H(t)$ remains closed under the Lie bracket.  
By the symmetry of  the operator $\Rm$, we will have $\im{\Rm(t)} \subset H(t)$ if and only if 
$K(t) \subset \ker\left(\Rm(t)\right)$, i.e., if $\Rm \circ \Ph(t) \equiv 0$ for all $t\in [0, T]$.
But if $\im{\Rm(t)} \subset H(t)$, it follows from Lemma \ref{lem:imagerm} that $H(t) \equiv \Hc$ 
and $K(t)\equiv \mathcal{K}$.  
To conclude from Lemma \ref{lem:ambrosesinger} that $\hol_p(g(t))\subset \Hc_p$, we need to know further that $H(t)=\Hc$ is
closed under parallel translation with respect to $\nabla_{g(t)}$ for all $t$.  However, by Lemma \ref{lem:vb}, this is true
if and only if $\nabla\Ph \equiv 0$ on $M\times [0, T]$.
Therefore, Theorem \ref{thm:rangerm} is a consequence of the following assertion.
\begin{theorem}\label{thm:rangermredux}
 Under the assumptions of Theorem \ref{thm:rangerm}, we have
\begin{equation}\label{eq:rangermredux}
\Rm\circ\Ph \equiv 0, \quad \nabla\Ph\equiv 0
\end{equation}
on $M\times [0, T]$ where $\Ph = \Ph(t)$ is the projection onto $K(t)$ with respect to $g(t)$ in the orthogonal decomposition
$\WM = H(t)\oplus K(t)$ provided by Lemma \ref{lem:proj}.
\end{theorem}

We remark that, given the dependence of the evolutions of $\nabla_{g(t)}$ and $\Ph(t)$ on the curvature, 
the aims of proving $\Rm\circ\Ph \equiv 0$ and $\nabla\Ph\equiv 0$ are not independent. We will
establish them simultaneously in the course of proving Theorem \ref{thm:rangermredux}.

\section{A PDE-ODE System}\label{sec:pdeode}

A few back-of-the-envelope calculations should convince the reader that the system consisting of $\Rh\defn\Rm\circ\Ph$ and
$\nabla \Ph$ is neither parabolic nor too far from being so. First, it is easy to see that the application 
of the heat operator to $\Rh$ produces
a term involving unmatched second derivatives of $\Ph$. Schematically,
\[
  \left(D_t- \Delta\right)\Rh = \left(\left(D_t - \Delta\right)\Rm\right)\ast \Ph + \nabla \Rm \ast \nabla \Ph 
  + \Rm \ast\Delta\Ph,
\]
where we use $V\ast W$ to denote some linear combination of contractions of the
tensors $V$ and $W$ by the metric. Since we have only defined $\Ph$ by the means of the fiber-wise ODE $D_t \Ph= 0$, we cannot
expect to have much control over $\nabla^{(k)}\Ph$ (beyond observations on the level of (1) of Lemma \ref{lem:vb}).
A natural option is to try to adjoin $\nabla\nabla\Ph$ itself to the system. 
This addition is logically redundant from the perspective of Theorem \ref{thm:rangermredux} since
$\nabla\Ph$ will be parallel on any time-slice on which $\Ph$ is parallel, but it comes at the cost
of introducing higher order curvature terms.  This can be seen from \eqref{eq:projode} and the standard formula
\[
  \pdt \Gamma_{ij}^k = -g^{mk}\left(\nabla_i R_{jm} + \nabla_j R_{im} - \nabla_m R_{ij}\right)	,
\]
for the evolution of the Christoffel symbols, which yield 
\[
\pdt \nabla\nabla\Ph = \nabla\nabla\Rm\ast \Ph + \nabla\Rm\ast \nabla\Ph.
\]
At minimum, we must introduce a component involving $\nabla\Rm$ to our system to compensate (as it turns out,
and unlike the second derivatives of $\Ph$, the factors of $\nabla\nabla\Rm$ may be controlled by regarding them, 
effectively, as factors of $\nabla(\nabla\Rm)$). 
From the perspective of Theorem \ref{thm:rangermredux}, the correct (i.e., redundant) such component ought to be
$\Th \defn \nabla \Rm \circ(\Id\times \Ph)$,
that is, the element of 
$TM^*\otimes \End(\WM)$ given by 
\[
\Th(X, \omega) \defn (\nabla_X\Rm)(\Ph(\omega)),
\]
since it must also vanish on any time slice where $\Rh \equiv 0$ and $\nabla \Ph\equiv 0$.
(In fact, for any $X$, the images of the endomorphisms $\nabla_X\Rm(g(p))$ lie in $\hol_p(g)$, cf. Remark 10.60 of \cite{Besse}.)
Fortunately, with this addition, our system stabilizes. The tensor $\nabla \Rm$ satisfies a heat-type
equation with reaction terms containing only products and contractions of $\Rm$ and $\nabla\Rm$: 
\[
 \left(\pdt - \Delta\right)\nabla\Rm = \nabla \Rm \ast\Rm,
\]
and the Laplacian falling on the composition $\Th$ generates only contractions of first- and second-covariant 
derivatives of $\Ph$ with $\nabla\Rm$ and $\nabla\nabla \Rm$.  
Thus we see that the application of the heat operator to $\Th$ introduces no fundamentally new quantities.
While we have been rather cavalier about the manner in which the components of the terms are combined (relative to the decomposition
 $H(t)\oplus K(t)$), we nevertheless are entitled to some optimism that the collection of $\Rh$,
$\Th$, $\nabla\Ph$, and $\nabla\nabla\Ph$ will fit into a closed system of mixed differential inequalities.
We will use the rest of this section to make this heuristic argument precise. 

\subsection{Notation and statement}
\label{ssec:notation}

In this section, we assume we have a solution to Ricci flow $g(t)$ and distributions $\Hc$ and $\Kc$ as in 
Theorem \ref{thm:rangerm}.  Let $H(t)$ and $K(t)$ be the distributions described in Lemma \ref{lem:proj},
and $\Pb(t)$, $\Ph(t)$ their associated projections.  We fix notation, once and for all, for the following collection
of tensors: 
\begin{equation*}
\begin{array}{llll}
  \Rb &\defn \Rm \circ \Pb,&\quad \Rh&\defn\Rm\circ\Ph,         \\
  \Tb&\defn \nabla\Rm\circ(Id\times\Pb),&\quad\Th&\defn \nabla\Rm\circ (Id\times\Ph),\\
    A &\defn \nabla \Ph,   &\quad  B &\defn \nabla\nabla\Ph.    
\end{array}
\end{equation*}

Note that $\Pb$ and $\Ph$ are self-adjoint elements of $E \defn \End(\WM)$.  
It will be convenient to use the metric identification
of $E\cong (\WM)^*\otimes \WM$ with $\WM\otimes\WM$ and further with the
subspace of $T_4(M)$ in which the members are antisymmetric in the first two
and last two arguments.  We make this identification by selecting the normalization
\[
  V \wedge W = \frac{1}{2}(V\otimes W - W \otimes V)
\]
for $V$, $W$ in $TM$ (or $T^*M$)
With respect to a local frame $\{e_a\}$ for $T^*M$, we have
\[
\Pb_{abcd} = \langle\Pb(e_a\wedge e_b), e_c\wedge e_d\rangle,
  \qquad \Ph_{abcd} = \langle\Ph(e_a\wedge e_b), e_c\wedge e_d\rangle,
\]
so, if $\omega \in \WM$,
\[
 \Pb(\omega)_{cd} = \Pb_{abcd}\omega_{ab}, \qquad \Ph(\omega)_{cd} = \Ph_{abcd}\omega_{ab}.
\]
We also define $T_{mabcd} = \nabla_{m}R_{abcd}$.

However, for the endomorphism $\Rm$, since we wish to keep the notation $R_{abcd}$ consistent with the usual convention 
(namely, that with respect to which one has $\langle \Rm(\omega), \omega\rangle \geq 0$ and $R_{abba} \geq 0$ on the standard 
sphere), we have an additional minus sign in our formula:
\[
  \Rm(\omega)_{cd} = -R_{abcd}\omega_{ab} = R_{abdc}\omega_{ab}.
\]
Similarly,
\[
  \left(\nabla_X\Rm\right)(\omega)_{cd} =-\nabla_mR_{abcd}X_m\omega_{ab} = T_{mabdc}X_m\omega_{ab}.
\]

The tensors $\Pb$ and $\Ph$, like $R$, are symmetric in the interchange of their first and last pairs of
indices and antisymmetric in the interchange of the elements of those pairs:
\begin{align*}
  \Pb_{abcd} &= \Pb_{cdab} = -\Pb_{abdc} = -\Pb_{bacd},\\
  \Ph_{abcd} &= \Ph_{cdab} = -\Ph_{abdc} = -\Ph_{bacd}.
\end{align*}
We also have $A_{mabcd} = \nabla_m \Ph_{abcd}$, $B_{mnabcd} = \nabla_m\nabla_n\Ph_{abcd}$, for which corresponding identities hold. The tensors $R$ and $T$
are of course, also subject to the Bianchi identities. 

The tensors $\Rb$, $\Rh$, $\Tb$, $\Th$
are no longer symmetric in the interchange of the final two pairs of indices, but remain antisymmetric
in the interchange of the elements of these pairs:
\begin{align}\label{eq:nonsym1}
\begin{split}
  \Rh_{ijkl} &= \Ph_{ijab}R_{ablk} = -\Rh_{jikl} = -\Rh_{ijlk},
\end{split}\\
\begin{split}\label{eq:nonsym2}
  \Th_{mijkl} &= \Ph_{ijab}T_{mablk} = -\Th_{mjikl} = -\Th_{mijlk},
\end{split}
\end{align}
and similarly for $\Rb$ and $\Tb$.

Now, we let $E \defn \End(\WM)$, 
\begin{align*}
  \mathcal{X} &\defn E\oplus\left(T^*M\otimes E\right) \cong T_4(M) \oplus T_5(M),\\
    \qquad \mathcal{Y} &\defn \left(T^*M\otimes E\right)\oplus\left(T^*M\otimes T^*M\otimes E\right)
	\cong T_5(M)\oplus T_6(M),
\end{align*}
and define 
\[
\ve{X}(t) \defn \Rh(t)\oplus \Th(t) \in \mathcal{X}
\]
and
\[ 
\ve{Y}(t)\defn A(t)\oplus B(t)\in \mathcal{Y}. 
\]

The goal of this section is to prove the following result.
 
\begin{proposition} \label{prop:pdeode}
With the above definitions, and under the assumptions of Theorem \ref{thm:rangerm},
for all $\delta > 0$, there exists a $C = C(n, K_0, \delta, T)$ such that on $M\times[\delta, T]$,
we have
\begin{align}
\label{eq:pdeodepar}
\left|\left(\pdt -\Delta_{g(t)}\right)\ve{X}\right|_{g(t)}^2 &
    \leq C\left(\left|\ve{X}\right|_{g(t)}^2  + \left|\ve{Y}\right|_{g(t)}^2\right)\\
\label{eq:pdeodeode}
\left|\pdt\ve{Y}\right|_{g(t)}^2 
&\leq C\left(\left|\ve{X}\right|_{g(t)}^2 + \left|\nabla\ve{X}\right|_{g(t)}^2 + \left|\ve{Y}\right|_{g(t)}^2\right).
\end{align}
\end{proposition}
Here we use the same notation to denote the metrics on $\mathcal{X}$, $\mathcal{Y}$, and $TM^*\otimes\mathcal{X}$ induced
by $g(t)$, and $\nabla = \nabla_{g(t)}$ and $\Delta_{g(t)}$ to denote the connection and Laplacian induced
on $\mathcal{X}$ by $g(t)$ and its Levi-Civita connection.

\begin{remark}
The parameter $\delta$ is an artifact of what will be an eventual application of Shi's estimates \cite{Shi}
for the derivatives of the curvature tensor, reflecting the degradation of the estimates as $t\to 0$.
If $M$ is compact, one can dispense with $\delta$ in favor of an estimate valid for all $t\in [0, T]$,
but with a constant $C$ that now also depends on the suprema of the norms of the first and second derivatives of curvature
on $M\times [0, T]$.
\end{remark}

\subsection{The orthonormal frame bundle associated to $g(t)$} \label{ssec:omt}

The verification of \eqref{eq:pdeodepar} and \eqref{eq:pdeodeode} will depend closely
on the algebraic structure of the evolution equations $\Rh$, $\Th$, $A$, $B$.  
To aid the computations, we will regard the tensors as functions on the product of the $g(t)$-orthonormal frame bundle $\OM$ with the interval $[0,T]$.
The utility of this perspective to calculations attached to the study of Ricci flow was first demonstrated 
by Hamilton in \cite{HamiltonHarnack}. For our application, we will borrow the notation
and abide by the conventions of Appendix F of \cite{RFV2P2}, thus, in particular, some commutation formulas
involving curvature will differ by a sign from their counterparts in \cite{HamiltonHarnack}.  

Following \cite{RFV2P2}, we let $\pi: \FM \to M$ denote the frame bundle of $M$ .  This is
a principal $\GL$-bundle on $M$; we take the group to act on the left.  
On $\gl$, one has the standard basis of elements
$\{e(a, b)\}_{a, b =1}^n$, with $e(a, b)_c^d = \delta^a_c\delta^d_b$.  We may fix a metric $h$ on  $\gl$
by insisting on the orthonormality of this basis with respect to $h$. Thus,
\[
      \langle e(a, b), e(c, d) \rangle_h = \delta_{d}^a\delta_b^c.
\]

Let $\mu: \GL\times \FM \to \FM$ denote the left action and, for any frame $Y$, define $\mu_Y: \GL \to \FM$ by
$\mu_Y(A) \defn \mu(A, Y)$. Then we have the isomorphism $(\mu_Y)_{*}:\gl \to T_Y(\FM_x)$ defining
the vertical spaces $\Vc_Y = \im{(\mu_Y)_{*}}$, where $\pi(Y) = x$.  At each $Y$, the Levi-Civita connection $\nabla$ of $g(t)$ defines 
complementary horizontal spaces $\Wc_Y \subset T_Y(\FM)$. For each $t\in [0, T]$, there is a unique metric 
$g^F(t)$ on $\FM$ which enforces the orthogonality of the subbundles $\Vc$ and $\Wc$ and
for which $\pi:  (\FM, g^F(t)) \to (M, g(t))$ is a submersion and
\[
(\mu_Y)_* : (\gl, h) \to \left(T_Y(F(M)_x), \left.g^F(t)\right|_{T_Y(F(M)_x)}\right)
\]
an isometry at each $Y\in \FM$.

A solution $g(t)$ to Ricci flow on $M\times (0, T)$ defines a map
\[
    g : \FM \times [0, T] \to SM_n(\mathbb{R}))
\]
with values in the symmetric $n\times n$ matrices.  Likewise, a time-dependent
family of sections of $T^k_l(M)$ may be regarded as an $O(n)$-equivariant matrix-valued
function on $\FMt \defn \FM \times [0, T]$.  
These functions are determined by their
values on the submanifold 
\[
  \OMt \defn g^{-1}(\Id) = \bigcup_{t\in [0, T]} \OM_{g(t)} \times \{t\} \subset \FMt,
\]
where $\OM_{g(t)} \subset \FM$ denotes the bundle of $g(t)$-orthogonal frames. 
It is convenient to use the same notation for the tensors under both of these interpretations.
Thus for $T\in T^1_{2}(M)$, we will write
\[
  T_{ab}^c = T(Y)_{ab}^c = T_{ab}^c(x) = T(x)(Y_a, Y_b, Y^c) 
\]
at a given $Y\in \FM$, 
where, again, $\pi(Y) = x$, and $Y^c\in T_x^*M$ is the $c$-th element of the frame dual to $Y$ at $x$.

\subsection{Elements of a global frame on $T\FMt$ and their commutators}
We continue to follow Appendix F of \cite{RFV2P2}.
From the isomorphisms $(\mu_Y)_*: \gl \to T_Y\FM_{\pi(Y)}$, we may generate a basis for each $\Vc_Y$
from $\{e(a, b)\}_{a, b = 1}^n$ by defining, for each $1\leq a,\; b \leq n$ and $Y\in \FM$,
\[
    \Lambda^a_b(Y) \defn (\mu_Y)_*e(a, b).
\]
The action of this vector field on a tensor is algebraic.  On $U\in T_2(M)$, for example, it is given by
\[
    \Lambda^{a}_b U_{ij} = \delta^a_i U_{bj} + \delta^a_{j} U_{ib},
\]
and on general $U\in T^k_l(M)$ by
\begin{align*}\begin{split}
    \Lambda^a_b U_{i_1i_2\ldots i_l}^{j_1 j_2\ldots j_k} &= 
    \delta_{i_1}^a U_{b i_2\ldots i_l}^{j_1 j_2\ldots j_k}
    +\delta_{i_2}^a U_{i_1 b\ldots i_l}^{j_1 j_2\ldots j_k} 
    + \cdots
    +\delta_{i_l}^a U_{i_1 i_2\ldots b}^{j_1 j_2\ldots j_k}\\
  &\phantom{=}
    -\delta_{b}^{j_1} U_{i_1 i_2\ldots i_l}^{a j_2\ldots j_k}
    -\delta_{b}^{j_2} U_{i_1 i_2\ldots i_l}^{j_1 a\ldots j_k}
    -\ldots
    -\delta_{b}^{j_k} U_{i_1 i_2\ldots i_l}^{j_1 j_2\ldots a}.	
\end{split}
\end{align*}

The collection $\{\La^a_b(Y)\}_{a, b=1}^n$ is an orthornormal basis for each $\Vc_Y$ with respect to 
(the restriction of) $g^F$, but these vector fields will not in general be parallel to $\OM$.
Thus it is sometimes convenient to consider instead the vectors
\begin{equation}
  \label{eq:vvectordef}
    \rho_{ab} \defn \delta_{ac}\Lambda^c_b - \delta_{bc}\Lambda^c_a.
\end{equation}
It is easily checked that the set $\{\rho_{ab}\}_{a < b}$ is an orthogonal basis for  $T_Y\OM_{\pi(x)}$.

Next we define a global frame spanning the horizontal subbundle $\Wc\subset T\FM$.  Given any $x\in M$, vector field $X \in T_xM$,
and frame $Y = (Y_1, Y_2, \ldots, Y_n) \in \FM_x$,  
define
\[
    \gamma_X(t) \defn (\tau_{\sg(t)}Y_1, \tau_{\sg(t)}Y_2, \ldots, \tau_{\sg(t)}Y_n) 
\]
where $\sigma(t)$ is any path in $M$ with $\sigma(0) = x$ and $\dot{\sigma}(t) = X$ and $
\tau_{\sg(t)}: T_{x}M \to T_{\sg(t)}M$ is parallel transport along $\sigma(t)$. 
Then we define, for $Y\in \FM$, and any $a = 1, 2, \ldots, n$,
\begin{equation}\label{eq:hvectordef}
 \left.\nabla_a\right|_Y \defn \left.\frac{d}{dt}\right|_{t=0} \gamma_{Y_a}(t).
\end{equation}
That is, we define $\left.\nabla_a\right|_Y$ to be the horizontal lift of $Y_a$ at $Y\in\FM_x$.

Local coordinates $\{x^i\}_{i=1}^n$ on $M$ define local coordinates $(\tilde{x}^i, y^i_a)$ on $\FM$
by $\tilde{x}^i = x^i\circ\pi$ and 
\[
       Y_a = y_a^i(Y)\pd{}{x^i}.
\]
In these coordinates, 
\[
    \nabla_a  = y^j_a(Y)\left(\pd{}{\tilde{x}^j} -y^k_b\Gamma^i_{kj}(\pi(Y))\pd{}{y^i_b}\right).
\]
Thus, for example, on a two-tensor $U$, 
\[
  \nabla_k(U_{ij}) = (\nabla U)_{kij} = \nabla_k U_{ij}
\]
where the leftmost expression represents the action of $\nabla_k\in T\FM$ on the $\R{n^2}$-valued function on $\FM$,
the middle expression represents the value of the $\R{n^3}$-valued function $\nabla U$ on $\FM$,
and the rightmost expression represents the tensor $\nabla U(x)$ evaluated at $Y_{k}(x)$, $Y_{i}(x)$, and $Y_{j}(x)$.
Where the interpretation is clear from the context, we will the notation of the
rightmost expression to represent all three cases.
The set $\{\left.\nabla_a\right|_{Y}\}_{a=1}^n$ is a basis for 
the horizontal space $\Wc_Y\subset T_Y\FM$ at
each $Y$. Since $\nabla_a g_{ij} = 0$, they are also tangent to $T_Y\OM$.

Finally we consider differentiation in the time direction.  As the vector $\pdt \in T\FMt$
is not in general tangent to $\OMt$, it is convenient to work instead with
the vector
\[
    D_t \defn \pdt + R_{ab}g^{bc}\Lambda_c^a
\]
which satisfies $D_t g_{ij} = 0$. On $\OMt$, it is given simply by
\[
   D_t  = \pdt + R_{ac}\Lambda^a_c.
\]
As remarked in Section 3, extending a tensor field $V$ defined on some time-slice 
to a time-dependent family via the ODE \eqref{eq:projode} is equivalent
to solving $D_t V \equiv 0$.  
In particular, for the projections $\Pb$ and $\Ph$, we have
\begin{equation}\label{eq:dtprojvanish}
  D_t \Pb_{abcd} \equiv D_t \Ph_{abcd} \equiv 0.
\end{equation}

The collection
$\{D_t\}\cup\{\nabla_a\}_{a=1}^n\cup \{\rho_{ab}\}_{1\leq a < b \leq n}$, forms the global frame field
for $T\OMt$ with respect to which we will perform our calculations (although it will be convenient to use of all elements of the set 
$\{\rho_{ab}\}_{1\leq a,\, b \leq n}$, i.e., including $\rho_{ab}$ for $ a\geq b$).  
As derivations on the frame bundle, they satisfy the following commutator relations.

\begin{lemma}\label{lem:commutator}
Restricted to $\OMt$, the vectors $D_t$, $\Lambda^a_b$, $\rho_{ab}$ and $\nabla_a$ satisfy
\begin{align}
\label{eq:lambdanablacomm}
[\Lambda^a_b, \nabla_c] &= \delta^b_c\nabla_a,\\
\label{eq:rhonablacomm}
[\rho_{ab}, \nabla_c] &= \delta_{ac}\nabla_b - \delta_{bc}\nabla_a,\\
\label{eq:nabladtcomm}
[D_t, \nabla_a] &= \nabla_bR_{ac}\rho_{bc} + R_{ac}\nabla_c = \nabla_pR_{pacb}\Lambda^b_c + R_{ac}\nabla_c,\\
\label{eq:heatnablacomm}
[D_t - \Delta, \nabla_a] &= R_{abdc}\nabla_b\rho_{cd} = 2R_{abdc}\Lambda^c_d\nabla_b + 2R_{ab}\nabla_b.
\end{align}
Here $\Delta = \nabla_p\nabla_p = \sum_{p=1}^n\nabla_p\nabla_p$.   
\end{lemma}
\begin{proof} Equations \eqref{eq:lambdanablacomm}, \eqref{eq:rhonablacomm}, and the
first equalities in \eqref{eq:nabladtcomm} and \eqref{eq:heatnablacomm} appear in 
Appendix F of \cite{RFV2P2}.  For the second equality in \eqref{eq:nabladtcomm}, we compute
\begin{align*}
  \nabla_bR_{ac}\rho_{bc} &= \nabla_bR_{ac}(\delta_{bp} \Lambda_c^p - \delta_{cp}\Lambda^p_b)\\
	      &= (\nabla_pR_{ac} - \nabla_c R_{ap})\Lambda_c^p\\
	      &= \nabla_pR_{pacb}\Lambda^b_c.
\end{align*}
For the second equality in \eqref{eq:heatnablacomm}, we compute
\begin{align*}
  R_{abdc}\nabla_b\rho_{cd} &= R_{abdc}\left(\rho_{cd}\nabla_b + [\nabla_b, \rho_{cd}] \right)\\
      &= R_{abdc}\left((\delta_{ce}\Lambda^e_d - \delta_{de}\Lambda^e_c)\nabla_b + (\delta_{db}\nabla_c - \delta_{cb}\nabla_d)\right)\\
      &= 2(R_{abdc}\Lambda_{c}^d +R_{ab})\nabla_b,
\end{align*}
using \eqref{eq:rhonablacomm} in the second line.
\end{proof}

\subsection{Evolution equations for $A$ and $B$.}

We begin by computing the evolution equations for the components of the ordinary-differential component of our
system. We will need the following consequence of Lemma \ref{lem:la}.
\begin{lemma}\label{lem:la2}
  The projections $\Pb$ and $\Ph$ satisfy
\begin{equation}\label{eq:la2}
  \Pb_{abcd}\Lambda_c^d\Ph_{ijkl} = 0.
\end{equation}
\end{lemma}
\begin{proof}
 Note that
\begin{align*}
   \Pb_{abcd}\Lambda_c^d\Ph_{ijkl} &= \Pb_{abdi}\Ph_{djkl} + \Pb_{abdj}\Ph_{idkl} 
				      + \Pb_{abdk}\Ph_{ijdl} + \Pb_{abdl}\Ph_{ijkd}\\
				   &= -\Pb_{abid}\Ph_{kldj} + \Pb_{abdj}\Ph_{klid}
				      - \Pb_{abkd}\Ph_{ijdl} + \Pb_{abdl}\Ph_{ijkd}\\
\begin{split}
	&=\left\langle\left[\Ph(e_k\wedge e_l), \Pb(e_a\wedge e_b)\right], e_i\wedge e_j\right\rangle\\
	&\phantom{=}\quad +\left\langle\left[\Ph(e_i\wedge e_j), \Pb(e_a\wedge e_b)\right], e_k\wedge e_l\right\rangle.
\end{split}  
\end{align*}
In view of Lemma \ref{lem:la}, we have $[K(t), H(t)] \subset K(t)$,
thus 
\begin{align*}
\begin{split}
   \Pb_{abcd}\Lambda_c^d\Ph_{ijkl}
  &=\left\langle\left[\Ph(e_k\wedge e_l), \Pb(e_a\wedge e_b)\right], \Ph(e_i\wedge e_j)\right\rangle\\
	&\phantom{=}\quad +\left\langle\left[\Ph(e_i\wedge e_j), \Pb(e_a\wedge e_b)\right], \Ph(e_k\wedge e_l)\right\rangle
\end{split}
\end{align*}
which vanishes on account of antisymmetry of the map $(X, Y, Z) \mapsto \langle [X, Y], Z\rangle$.
\end{proof}

In view of \eqref{eq:dtprojvanish}, the only non-zero contributions to the evolution equations for $A = \nabla\Ph$ and 
$B = \nabla\nabla \Ph$ come from the time-dependency of the connection. These contributions 
are encoded in the commutators of $D_t$ with the horizontal vectors $\nabla_a$.

\begin{proposition}\label{prop:aevol}
Regarded as a matrix-valued function on $\OMt$, the tensor $A$ evolves according
to
\begin{align}\label{eq:aevol}
\begin{split}
D_t A_{mijkl} &=  R_{mr}A_{rijkl} -\Ph_{pjkl}\Th_{rpirm} - \Ph_{ipkl}\Th_{rpjrm}\\
&\phantom{=}- \Ph_{ijpl}\Th_{rpkrm} - \Ph_{ijkp}\Th_{rplrm}.
\end{split}
\end{align}

\end{proposition}
\begin{proof}
  Since $D_t\Ph_{ijkl} = 0$, we have 
  $D_t A_{mijkl} = D_t\nabla_m\Ph_{ijkl} = [D_t, \nabla_m]\Ph_{ijkl}$.
Thus from \eqref{eq:nabladtcomm}
we have
\[
 D_t A_{mijkl} =  R_{mr}A_{rijkl} + T_{rrmpq}\Lambda^q_p\Ph_{ijkl}.
\]
Now,
\begin{align*}
  T_{rrmpq} &= T_{rrmuv}(\Ph_{uvpq} + \Pb_{uvpq})\\
	    &= -\Th_{rpqrm} + T_{rrmuv}\Pb_{uvpq},
\end{align*}
so
\begin{align*}
  T_{rrmpq}\Lambda^q_p\Ph_{ijkl} &= -\Th_{rpqrm}\Lambda^q_p\Ph_{ijkl} + T_{rrmuv}\Pb_{uvpq}\Lambda^q_p\Ph_{ijkl}\\
  &=-\Th_{rpqrm}\Lambda^q_p\Ph_{ijkl},
\end{align*}
on account of Lemma \ref{lem:la2}, and \eqref{eq:aevol} follows.
\end{proof}
\begin{proposition}\label{prop:bevol}
Regarded as a matrix-valued function on $\OMt$, the tensor $B$ evolves according
to
\begin{align}
\label{eq:bevol}
\begin{split}
 &D_t B_{mnijkl} = R_{mr}B_{rnijkl} + R_{nr}B_{mrijkl}   + \nabla_n R_{ms} A_{sijkl}\\
		&\quad\phantom{=}  + T_{rrmsi}A_{nsjkl} + T_{rrmsj}A_{niskl} + T_{rrmsk}A_{nijsl} + T_{rrmsl}A_{nijks} \\
		&\quad\phantom{=} -\Ph_{sjkl}\nabla_m\Th_{rsirn}-\Ph_{iskl}\nabla_m\Th_{rsjrn}-\Ph_{ijsl}\nabla_m\Th_{rskrn}
			      -\Ph_{ijks}\nabla_m\Th_{rslrn}\\
		&\quad\phantom{=} +T_{rrnvw}(\Ph_{sjkl}A_{mvwsi}+\Ph_{iskl}A_{mvwsj}+\Ph_{ijsl}A_{mvwsk}+\Ph_{ijks}A_{mvwsl}).
\end{split}
\end{align}
\end{proposition}
\begin{proof}
As before, $D_t B_{mnijkl} = [D_t, \nabla_m\nabla_n] \Ph_{ijkl}$. We compute this commutator
using a double application of \eqref{eq:nabladtcomm}: 
\begin{align*}
  [D_t, \nabla_m\nabla_n] & = [D_t, \nabla_m]\nabla_n + \nabla_m [D_t, \nabla_n]\\
    &=\left(R_{mr}\nabla_r + T_{rrmsu}\Lambda^u_s\right)\nabla_n 
      +\nabla_m\left(R_{nr}\nabla_r + T_{rrnsu}\Lambda^u_s\right)\\
\begin{split}
    &= R_{mr}\nabla_r\nabla_n + T_{rrmsu}\Lambda^u_s\nabla_n
      + \nabla_mR_{nr}\nabla_r + R_{nr}\nabla_m\nabla_r\\
    &\phantom{=} + \nabla_m T_{rrnsu}\Lambda^u_s + T_{rrnsu}\nabla_m\Lambda^u_s.
\end{split}
\end{align*}
Then, using \eqref{eq:lambdanablacomm}, we have
$T_{rrnsu}\nabla_m\Lambda^u_s = -T_{rrnsm}\nabla_s = T_{rrnms}\nabla_s$,
so
\begin{align}\label{eq:bevolp1}
\begin{split}
 D_t B_{mnijkl} &= R_{mr}B_{rnijkl} + R_{nr}B_{mrijkl} +  \nabla_m T_{rrnsu}\Lambda^u_s\Ph_{ijkl}\\
		&\phantom{=}  + (\nabla_m R_{ns} + T_{rrnms})A_{sijkl}+ T_{rrmsu}\Lambda^u_sA_{nijkl}. 
\end{split}
\end{align}

Now, since $\Pb + \Ph = \Id$, we have $A = \nabla\Ph = -\nabla\Pb$;
applying this and considering the decomposition of $T$ into components as above,
we compute
\begin{align*}
 \nabla_m T_{rrnsu} &= \nabla_m( -\Th_{rsurn} + T_{rrnvw}\Pb_{vwsu})\\
		    &= - \nabla_m\Th_{rsurn}  - T_{rrnvw}A_{mvwsu} + \nabla_mT_{rrnvw}\Pb_{vwsu}.\\
\end{align*}
Using Lemma \ref{lem:la2} again, we therefore have
\begin{equation}\label{eq:bevolp2}
 \nabla_m T_{rrnsu}\Lambda^u_s\Ph_{ijkl} = - (\nabla_m\Th_{rsurn} + T_{rrnvw}A_{mvwsu})\Lambda^u_s\Ph_{ijkl}.
\end{equation}

Finally, we can simplify the last line of \eqref{eq:bevolp1}.  The last term is
\begin{align*}\begin{split}
  T_{rrmsu}\Lambda^u_sA_{nijkl} &=
    T_{rrmsn}A_{sijkl} + T_{rrmsi}A_{nsjkl} + T_{rrmsj}A_{niskl}\\
    &\phantom{=} + T_{rrmsk}A_{nijsl} + T_{rrmsl}A_{nijks}, 
\end{split}
\end{align*}
and
\begin{align*}
  \nabla_m R_{ns} + T_{rrnms} + T_{rrmsn} &= \nabla_m R_{ns} + (\nabla_s R_{mn} - \nabla_mR_{sn}) +
	  (\nabla_n R_{sm} - \nabla_s R_{nm})\\
  &= \nabla_{n} R_{ms},
\end{align*}
so the last line reduces to 
\begin{align}\label{eq:bevolp3}
\begin{split}
  &(\nabla_m R_{ns} + T_{rrnms})A_{sijkl}+ T_{rrmsu}\Lambda^u_sA_{nijkl}\\
  &\quad= \nabla_n R_{ms} A_{sijkl} + T_{rrmsi}A_{nsjkl} + T_{rrmsj}A_{niskl}\\
  &\quad\phantom{=} + T_{rrmsk}A_{nijsl} + T_{rrmsl}A_{nijks}.
\end{split}
\end{align}
Combining \eqref{eq:bevolp1}, \eqref{eq:bevolp2}, and\eqref{eq:bevolp3}, we then obtain
\eqref{eq:bevol}.
\end{proof}

\subsection{Evolution equations for $R$ and $T$.}

Recall that for $A$, $B\in \End(\WMp)$, one can form the product $A\#B \in \End(\WMp)\cong \WMp^*\otimes\WMp^*$
defined by
\[
  A\#B(\omega) \defn \frac{1}{2}\sum_{M,N} \left\langle [A(\varphi^{M}), B(\varphi^{N})], \omega \right\rangle 
			    \cdot [\varphi^{M}, \varphi^{N}]
\]
where $\{\varphi_{\alpha}\}$ is an orthonormal basis for $\WMp$.
This product is bilinear and symmetric in its arguments, and with it, we define the square $A^{\#} \defn A\#A$.  
In terms of the structure constants $[\varphi^M, \varphi^N] = C^{MN}_P\varphi^{P}$ 
(and regarded as an element of $\WMp^*\otimes\WMp^*$),
we have $(A\#B)_{IJ} =(1/2)A_{MP}B_{NQ}C^{PQ}_IC^{MN}_J$.

Now define 
\[
 \Qc: \End(\WM)\to \End(\WM)
\]
by
\begin{equation}\label{eq:qcdef}
  \Qc(A) \defn A^2 + A^{\#}
\end{equation}
and
\[
\Sc: \End(\WM)\times (TM^*\otimes\End(\WM))  \to TM^*\otimes \End(\WM)
\]
by
\begin{equation}\label{eq:scdef}
  \Sc(A, F)(X, \cdot) \defn A\circ (F\lrcorner X) + (F\lrcorner X)\circ A + 2 (F\lrcorner X) \# A. 
\end{equation}

These operators arise as reaction terms in the evolution equations for $R$ and $T$.

\begin{proposition}
  Viewed as matrix-valued functions on $\OMt$, the tensors $R$ and $T$ evolve
  according to
\begin{align}\label{eq:revol}
(D_t - \Delta)R_{ijkl} &= -\Qc(\Rm)_{ijkl}, 
\end{align}
and
\begin{align}\label{eq:tevol}
(D_t - \Delta)T_{mijkl} &= 2R_{mb}T_{bijkl}+ 2R_{mbdp}\Lambda^p_dT_{bijkl}  - \Sc(\Rm, \nabla \Rm)_{mijkl}.
\end{align}
\end{proposition}
\begin{proof}
Equation \eqref{eq:revol} is standard.  For \eqref{eq:tevol}, we use \eqref{eq:revol} and \eqref{eq:heatnablacomm}:
\begin{align*}
  (D_t - \Delta) T_{mijkl} &= [D_t - \Delta, \nabla_m]R_{ijkl} + \nabla_m(D_t - \Delta)R_{ijkl}\\
			   &= 2\left(R_{mb}\nabla_b + R_{mbdp}\Lambda^p_d\nabla_b\right)R_{ijkl}
			   - \nabla_m(\Qc(\Rm))\\
			   &= 2R_{mb}T_{bijkl} +2R_{mbdp}\Lambda^p_dT_{bijkl} 
			      - \nabla_m(\Qc(\Rm)).
\end{align*}
For the last term, note that at any $p\in M$, 
\begin{align*}
  \nabla (\Qc(\Rm))(e_m, \cdot) &= \nabla_m\Rm \circ \Rm + \Rm\circ\nabla_m \Rm 
      + \Rm\#\nabla_m \Rm\\
			      &= \Sc(\Rm,\nabla\Rm)(e_m, \cdot)
\end{align*}
in view of the symmetry of the product $\#$.
\end{proof}

\begin{remark}
 We choose to leave the terms $2R_{mb}T_{bijkl} +2R_{mbdp}\Lambda^p_dT_{bijkl}$ in \eqref{eq:tevol} in a rather raw form
 for convenience in a later computation, however, we might alternatively have written
\begin{align*}\begin{split}
  R_{mb}T_{bijkl} +R_{mbdp}\Lambda^p_dT_{bijkl} &= R_{mbdi}T_{bdjkl} + R_{mbdj}T_{bidkl}+ R_{mbdk}T_{bijdl} + R_{mbdl}T_{bijkd}
  \end{split}\\
	    &\defn\Uc(\Rm, \nabla\Rm)_{mijkl}
\end{align*}
where
\[
 \Uc: \End(\WM)\times (TM^*\otimes\End(\WM))  \to TM^*\otimes \WM\otimes\WM
\]
is given by
\[
  \Uc(A, F)(X, \omega, \eta) = \sum_{i =1}^n\left(\left\langle\left[A(e_i\wedge X), F(e_i, \omega) \right], \eta\right\rangle
			   + \left\langle\left[A(e_i\wedge X), F(e_i, \eta)\right], \omega\right\rangle \right)
\]
in the fiber over $p$ for $\{e_i\}$ an orthonormal basis of $T_pM$.
Alternatively, using the second Bianchi identity and the symmetries of $T$, one can define the tensor 
$C_{mijkl} \defn -T_{mipqj}R_{kpql}$ (analogous to Hamilton's $B_{ijkl} = - R_{ipqj}R_{kpql}$) and write the evolution
of $T$ in the form
\begin{align*}
\begin{split}
  &(D_t - \Delta) T_{mijkl}\\
  &\quad=2(C_{mijkl} + C_{mklij} - C_{mijlk} - C_{mlkij} + C_{mikjl} +  C_{mjlik} - C_{miljk} - C_{mjkil})\\
    &\quad\phantom{=} + 2(C_{kljmi} -  C_{lkjmi} + C_{lkimj} - C_{klimj} + C_{ijlmk} -  C_{jilmk} + C_{jikml} - C_{ijkml}). 
\end{split}
\end{align*}
If $\mathcal{P}$ denotes the projection $T_4(M)\to\WM\otimes_S\WM$ (where $\otimes_S$ denotes
the symmetric tensor product), that is,
\[
  \mathcal{P}(V)_{ijkl} = \frac{1}{8}(V_{ijkl} - V_{jikl} - V_{ijlk} + V_{jikl} + V_{klij} - V_{lkij} - V_{klji} + V_{lkji}),
\]
then the sum in parantheses on the last line in the expression above (which corresponds to $\Uc(\Rm, \nabla\Rm)$) is 
$-8\cdot\mathcal{P}(C\lrcorner_4 e_m)$
where $\lrcorner_i X$ denotes inner multiplication by $X$ in the $i$-th argument.
\end{remark}

\subsection{Evolution equations for $\Rh$ and $\Th$.}

Using the results of the preceding section, we now compute the evolutions of the components of the parabolic portion
of our PDE-ODE system. We begin with the consideration of the reaction terms $\Qc(\Rm)$ and $\Sc(\Rm, \nabla\Rm)$.

\begin{lemma}\label{lem:rxnterms}
 Denote temporarily $R = \Rm$, $T = \nabla\Rm$. At any $p \in M$, 
\begin{equation}\label{eq:qcomp}
 \Qc(R) \circ \Ph = R\circ \Rh + \Rb^*\#\Rh^* + \Rm\#\Rh^* 
\end{equation}
and 
\begin{equation}\label{eq:scomp}
  (\Sc(R, T)\lrcorner X)\circ \Ph =  R\circ \Th_X + T_X\circ \Rh + (\Th_X^*\#\Rb^* + T_X\#\Rh^*)\circ\Ph   
\end{equation}
for any $X \in T_pM$, where we use the shorthand 
\[
T_X\defn T\lrcorner X \quad \Th_X\defn \Th \lrcorner X\in \End(\WMp),
\]
and denote the adjoint of an operator $A\in \End(\WMp)$ by $A^*$. 
\end{lemma}
\begin{remark}  As we observed in \eqref{eq:nonsym1} and \eqref{eq:nonsym2}, the operators $\Rb$, $\Rh$, $\Tb$, $\Th$ are no longer 
self-adjoint.  However, $\Rb^* = \Pb \circ R$, $\Rh^* = \Ph\circ R$, $\Tb^*_X = \Pb\circ T_X$, $\Th_X^* = \Ph\circ T_X$.  In coordinates,
for example, $\Rb^*_{ijkl} = \Pb_{abkl}R_{ijba}$, and similarly for the others.
\end{remark}
\begin{proof}
  The origin of the first term in \eqref{eq:qcomp} is clear.  For the second, we use $\Pb +\Ph = \Id$ and expand
to find
\begin{align*}
    R\#R &= \left((\Ph + \Pb)\circ R\right) \# \left((\Ph + \Pb)\circ R\right)\\
	 &=
	  (\Pb\circ R)\# (\Pb\circ R) + (\Ph\circ R)\# (\Pb\circ R) 
		+ (\Pb\circ R)\# (\Ph\circ R) + (\Ph\circ R)\# (\Ph\circ R)\\
	 &=  \Rb^*\# \Rb^* + 2\Rb^*\# \Rh^* + \Rh^*\#\Rh^*.
\end{align*}
We claim $(\Rb^*\# \Rb^*)\circ\Ph \equiv 0$.  To see this, let $p\in M$ and $\{\varphi^A\}$ be an orthogonal
basis for $\wedge^2(T_pM)$. Then for any $1\leq N \leq n(n-1)/2$, we have
\begin{align*}
  (\Rb^*\# \Rb^*)(\Ph(\varphi^N)) &=
      \frac{1}{2}\sum_{A, B}\left\langle \left[\Pb(R(\varphi^A)), \Pb(R(\varphi^B))\right], \Ph(\varphi^N)\right\rangle
	\cdot[\varphi^A, \varphi^B]\\
      &= \frac{1}{2}\sum_{A, B}\mathcal{T}[\Pb, \Pb, \Ph](R(\varphi^A), R(\varphi^B), \varphi^N)\cdot [\varphi^A, \varphi^B]
\end{align*}
which vanishes by Lemma \ref{lem:vb}. Thus  
\begin{align*}
  (R^2 + R^\#)\circ\Ph &=  R\circ\Rh + (2\Rb^*\# \Rh^* + \Rh^*\#\Rh^*)\circ\Ph\\
		       &= R\circ\Rh + (\Rb^*\# \Rh^* + R\#\Rh^*)\circ\Ph
\end{align*}

We argue similarly for \eqref{eq:scomp}. Again the first two terms are clear, and for the term involving the Lie-algebraic product,
we expand into components relative to the decomposition $\WMp = H_p \oplus K_p$:
\begin{align*}
  T_X\# R &= (\Tb_X^* + \Th_X^*)\#(\Rb^*+ \Rh^*)\\
	  &=  \Tb_X^* \#\Rb^* + \Th_X^* \#\Rb^* + T_X \#\Rh^*.
\end{align*}
Then, just as before, 
\[
  (\Tb_X^*\#\Rb^*)\circ\Ph(\varphi^N) = 
  \frac{1}{2}\sum_{A, B}\mathcal{T}[\Pb, \Pb, \Ph](T(X, \varphi^A), R(\varphi^B), \varphi^N)\cdot[\varphi^A, \varphi^B],
\]
and so is zero for all $N$ by Lemma \ref{lem:vb}. 
\end{proof}
\begin{remark}
  As functions on the frame bundle, we have
\begin{align}\label{eq:qpframe}
\begin{split}
  (\Qc(\Rm)\circ\Ph)_{ijkl} &= R_{cdlk}\Rh_{ijcd} 
      + 2\Ph_{ijce}(\Rb^*_{cpql}\Rh_{epqk} - \Rb^*_{cpqk}\Rh_{epql})\\
      &\phantom{=}   + 2\Ph_{ijce}(R_{cpkq}\Rh^*_{epql} -R_{cplq}\Rh^*_{epqk}),
\end{split}
\end{align}
and
\begin{align}\label{eq:spframe}
\begin{split}
  (\Sc(\Rm, \nabla T)\circ(\Id\times\Ph))_{mijkl} &= R_{ablk} \Th_{mijab} +T_{mablk}\Rh_{ijab} \\
  &\phantom{=} + 2\Ph_{ijce}(\Rb^*_{cpql}\Th^*_{mepqk} - \Rb^*_{cpqk}\Th_{mepql})\\
      &\phantom{=}   + 2\Ph_{ijce}(T_{mcpkq}\Rh^*_{epql} -T_{mcplq}\Rh^*_{epqk}).
\end{split} 
\end{align}

\end{remark}

\begin{proposition}\label{prop:rhatevol}
The tensor $\Rh$, regarded as a matrix-valued function on $\OMt$, evolves according
to
\begin{align}\label{eq:rhatevol}
    (D_t - \Delta)\Rh_{ijkl} &= 2A_{pijab}T_{pabkl} + B_{ppijab}R_{abkl} + (\Qc(\Rm)\circ\Ph)_{ijkl}.
\end{align}
\end{proposition}
\begin{proof}
We have 
\begin{align*}
  (D_t - \Delta)\Rh &= \Rm\circ(D_t- \Delta)\Ph - 2\nabla_p\Rm\circ\nabla_p\Ph  + (D_t - \Delta)\Rm \circ\Ph \\
		    &= -\Rm\circ\Delta\Ph - 2\nabla_p\Rm\circ\nabla_p\Ph + \Qc(\Rm)\circ\Ph,
\end{align*}
using \eqref{eq:dtprojvanish} and \eqref{eq:revol}.  Then
\[
-(\Rm\circ\Delta\Ph)_{ijkl} = -\Delta\Ph_{ijab}R_{ablk} = B_{ppijab}R_{abkl},
\]
and
\[
  - 2(\nabla_p\Rm\circ\nabla_p\Ph)_{ijkl} = -2\nabla_p\Ph_{ijab}\nabla_pR_{ablk} = 2A_{pijab}T_{pabkl},
\]
and \eqref{eq:rhatevol} follows.
\end{proof}

\begin{proposition}
The tensor $\Th$, viewed as a matrix-valued function on $\OMt$, evolves according to
\begin{align}\label{eq:thevol}
\begin{split}
    &(D_t - \Delta)\Th_{mijkl}\\
    &\quad= 2A_{pijab}\nabla_pT_{mijkl} + B_{ppijab}T_{mabkl}
      +\left(\Sc(R, T)\circ(\Id\times\Ph)\right)_{mijkl}\\
   &\quad\phantom{=} +2\left(R_{mpqi}\Th_{pqjkl}+R_{mpqj}\Th_{piqkl}
 	  + R_{mpqk}\Th_{pijql} + R_{mpql}\Th_{pijkq}\right)\\
 &\quad\phantom{=} + 2\left(\Rh_{qimp}\Th_{pqjkl} + \Rh_{qjmp}\Th_{piqkl} 
 	  + T_{pablk}\left(\Ph_{ijqb}\Rh_{qamp}+ \Ph_{ijaq}\Rh_{qbmp}\right)\right).    
\end{split}
\end{align}
\end{proposition}
\begin{proof}
We obtain the evolution equation for $\Th$ by a computation similar to that for $\Rh$. Namely, we have
$\Th_{mijkl} = -\Ph_{ijab}T_{mabkl}$, and, as before, 
\begin{equation}\label{eq:thevol1}
  (D_t - \Delta)\Th_{mijkl} = \Ph_{ijab}(D_t - \Delta)T_{mablk} + 2A_{pijab}\nabla_pT_{mabkl} + B_{ppijab}T_{mabkl}.
\end{equation}
By \eqref{eq:tevol}, we have
\begin{align}\label{eq:thevol2}
\begin{split}
  (D_t - \Delta) T_{mablk} &=2(R_{mp} + R_{mpqr}\Lambda^r_q)T_{pablk} -\mathcal{S}(\Rm, \nabla\Rm)_{mablk}, \\
\end{split}
\end{align}
and
\[
\mathcal{S}(\Rm, \nabla\Rm)_{mablk}\Ph_{ijab} = -(\mathcal{S}(R, T) \circ(Id\times \Ph))_{mijkl},
\]
so we just need to consider the contraction of first term in \eqref{eq:thevol2} against $\Ph_{ijab}$.  

First, since $\Lambda_{q}^r$ is a derivation,
we can write
\begin{align*}
 \Ph_{ijab}\Lambda^r_qT_{mablk} = \Lambda^r_q\Th_{mijkl} - T_{mablk}\Lambda^r_q\Ph_{ijab}.
\end{align*}
Also, 
\[
 R_{mpqr} = R_{qrmp} = (\Ph_{qruv} + \Pb_{qruv})R_{uvmp} = \Rh_{qrpm} + \Pb_{uvqr}R_{uvmp},
\]
so
\begin{align*}
 \Ph_{ijab}R_{mpqr}\Lambda^r_q T_{mablk} &= R_{mpqr}\Lambda^r_q\Th_{mijkl} - T_{mablk}(\Rh_{qrpm} + \Pb_{uvqr}R_{uvmp})\Lambda^r_q\Ph_{ijab}\\
      &= R_{mpqr}\Lambda^r_q\Th_{mijkl} - T_{mablk}\Rh_{qrpm}\Lambda^r_q\Ph_{ijab}
\end{align*}
by Lemma \ref{lem:la2}.
With this, we can expand to obtain
\begin{align}\label{eq:thevol3}
\begin{split}
  &\Ph_{ijab}(R_{mp} + R_{mpqr}\Lambda^r_q)T_{pablk}\\ 
&\quad= R_{mpqi}\Th_{mqjkl} + R_{mpqj}\Th_{miqkl}+R_{mpqk}\Th_{mijkql} +R_{mpql}\Th_{mijkq}\\
&\quad\phantom{=}-T_{mablk}\left(\Ph_{qjab}\Rh_{qipm} + \Ph_{iqab}\Rh_{qjpm} 
			    + \Ph_{ijqb}\Rh_{qapm} +\Ph_{ijaq}\Rh_{qbpm}\right)\\
&\quad= R_{mpqi}\Th_{mqjkl} + R_{mpqj}\Th_{miqkl}+R_{mpqk}\Th_{mijkql} +R_{mpql}\Th_{mijkq}\\
&\quad\phantom{=}-\Th_{mqjkl}\Rh_{qipm} - \Th_{miqkl}\Rh_{qjpm} 
			    + T_{mabkl}\left(\Ph_{ijqb}\Rh_{qapm} + \Ph_{ijaq}\Rh_{qbpm}\right).
\end{split}
\end{align}
Equations \eqref{eq:thevol1}, \eqref{eq:thevol2}, and \eqref{eq:thevol3} then combine 
to yield \eqref{eq:thevol}.
\end{proof}

\begin{remark}
  For the sequel, we observe that the quantities $A$, $B$, $\Rh$, and $\Th$ satisfy the following schematic
  equations:
\begin{align}
\label{eq:schemabevol}
  D_t A &= R\ast A + \Ph\ast\Th,\\
\label{eq:schembbevol}
\begin{split}
  D_t B &= R\ast B + T\ast A + \Ph\ast\nabla\Th + \Ph\ast T\ast A,
\end{split}\\
\label{eq:schemrhevol}
  \begin{split}
  (D_t - \Delta) \Rh &= T\ast A + R\ast B + R\ast\Rh + \Ph\ast\Rb^*\ast\Rh + \Ph \ast R\ast \Rh,
\end{split}\\
\label{eq:schemthevol}
\begin{split}
  (D_t - \Delta) \Th &= \nabla T \ast A + R\ast B + R\ast \Rh + \Ph\ast\Rb^*\ast\Rh + \Ph\ast R\ast \Rh\\
     &\phantom{=}+\Ph\ast T\ast \Rh + \Ph\ast\Rb^*\ast\Th^*.
\end{split}
\end{align}

For our purposes, the key feature of these equations is that each term contains at least one factor
  of (some contraction of) $A$, $B$, $\Rh$, $\Th$, $\nabla \Th$, or their adjoints.  
Under our hypotheses, the other factors (including
the extra linear factors of the components of our system) will be bounded,
and this is enough for the application of the backwards-uniqueness result from \cite{Kotschwar}, 
Theorem \ref{thm:bu}, below.
\end{remark}

\subsection{Proof of Proposition \ref{prop:pdeode}}
We are now in a position to prove Proposition \ref{prop:pdeode}.  By the estimates of Shi \cite{Shi}, 
if $g(t)$ is a complete solution to \eqref{eq:rf} with 
\[
|\Rm(x, t)|_{g(t)} \leq K_0
\] 
on $M\times [0, T]$, then, for all $m\geq 1$, and all $\delta > 0$, there exist constants $K_m = K_m(n, K_0, T, \delta)$ such that
\begin{equation}\label{eq:rmderbound}
  \left|\nabla^{(m)}\Rm(x, t)\right|_{g(t)} \leq K_m
\end{equation}
on $M\times [\delta, T]$. The tensors $\Pb$ and $\Ph$, being projection tensors, are also clearly bounded. In fact,
 if $\dim{\mathcal{H}} = k$, then $|\Pb|^2_{g(t)} \equiv k$ and $|\Ph|^2_{g(t)} \equiv n(n-1)/2 - k$ 
on $M\times[0, T]$.  Hence $\Rb$, $\Rh$, $\Tb$, and $\Th$
(and their adjoints) are likewise uniformly bounded on $M\times [\delta, T]$.  Thus we have only to verify that $A$ and $B$
are also bounded.  This is more or less evident from the evolution equations 
\eqref{eq:aevol}, \eqref{eq:bevol} at this point. We only need to observe first that, 
since $D_t = \pdt + R_{ab}\Lambda^a_b$, one has
\[
  D_t U = \pdt U + \Rc\ast U,
\]
for any tensor $U$ and then, that, from \eqref{eq:rmderbound} and the above discussion, we have, on $M\times [\delta, T]$,
\[
 \left|\pdt A\right| \leq C(|A| + 1), \quad
  \left|\pdt B\right| \leq C(|A| + |B|+ 1),
\]
for an appropriate $C$.  (Note that $\nabla \Th = A \ast T + \Ph \ast \nabla T$.) 
At $t = T$, we have $|A| = |B| = 0$, 
so we obtain that $|A|$ and (consequently) $|B|$ are bounded on $M\times [\delta, T]$ as well.

Taken with equations 
\eqref{eq:schemabevol} - \eqref{eq:schemrhevol}, we have established that there exists a constant
$C = C(n, K_0, T, \delta)$ such that
\begin{align}
  \left|\pdt A \right|_{g(t)} &\leq C\left(|A|_{g(t)} + |\Th|_{g(t)}\right),\\
  \left|\pdt B \right|_{g(t)} &\leq C\left(|A|_{g(t)} + |B|_{g(t)} + |\nabla\Th|_{g(t)}\right),\\
  \left|\left(\pdt - \Delta\right)\Rh\right|_{g(t)} &\leq
    C\left(|A|_{g(t)} + |B|_{g(t)} + |\Rh|_{g(t)}\right),\\
  \left|\left(\pdt - \Delta\right)\Th\right|_{g(t)} &\leq C\left(|A|_{g(t)} + |B|_{g(t)} 
		+ |\Rh|_{g(t)} + |\Th|_{g(t)}\right),
\end{align}
on $M\times [\delta, T]$. Proposition \ref{prop:pdeode} then follows at once from the Cauchy-Schwarz inequality.

\section{Backwards-uniqueness of the PDE-ODE system}\label{sec:bu}

The following is a special case of Theorem 3.1 in \cite{Kotschwar}.

\begin{theorem}\label{thm:bu}
Let $\mathcal{X}$ and $\mathcal{Y}$ be finite direct sums of the bundles $T^k_l(M)$,
and 
$X \in C^{\infty}(\mathcal{X}\times[A, \Omega])$, $Y\in C^{\infty}(\mathcal{Y}\times[A, \Omega])$.
Suppose $g(t)$ is a smooth, complete solution to \eqref{eq:rf} of uniformly bounded curvature.
Further assume that the sections $X$, $Y$, and $\nabla X$ are uniformly bounded with respect to $g(t)$ and satisfy
\begin{align}
\label{eq:pdeineq}
\left|\left(\pdt - \Delta_{g(t)}\right)X\right|_{g(t)}^2 &\leq C \left(|X|_{g(t)}^2 + |\nabla X|_{g(t)}^2 + |Y|_{g(t)}^2\right),\\
\label{eq:odeineq}
\left|\pd{Y}{t}\right|_{g(t)}^2 &\leq C\left(|X|_{g(t)}^2 + |\nabla X|^2_{g(t)} + |Y|_{g(t)}^2\right)
\end{align}
for some $C \geq 0$.
Then $X(\cdot, \Omega)\equiv 0$, $Y(\cdot, \Omega) \equiv 0$ implies $X \equiv 0$, $Y\equiv 0$ on
$M\times [A, \Omega]$. 
\end{theorem}

Combining this result with Proposition \ref{prop:pdeode}, we have essentially proven Theorem \ref{thm:rangermredux};
it only remains to see that the conclusion is valid all the way down to $t= 0$.

\begin{proof}[Proof of Theorem \ref{thm:rangermredux}]
With $\mathcal{X}$, $\mathcal{Y}$, and $X(t)$, $Y(t)$ defined as in the previous section, we may apply
Proposition \ref{prop:pdeode} and Theorem \ref{thm:bu} on $M\times [\delta, T]$ for any $0 < \delta < T$,
to obtain the conclusion of Theorem \ref{thm:rangermredux} (and hence Theorem \ref{thm:rangerm}) for all  
$t \in (0, T)$.  But $\Pb(t)$ and $\Ph(t)$ are smoothly defined (and are complementary $g(t)$-orthogonal 
projections) on $\WM$ for all $t\in [0, T]$.  Thus
the vanishing of $\nabla \Pb$ and $\nabla \Ph$ on $M\times (0, T)$ imply by continuity that
$\nabla \Pb(0) = \nabla \Ph(0) =0$ also.
Moreover, $\ker (\Pb(t)) \equiv \mathcal{K}$ and $\ker (\Ph(t)) \equiv \mathcal{H}$ for $t\in (0, T)$,
thus continuity again implies that $\mathcal{K}\subset\ker\Pb(0)$ and $\Hc\subset \ker\Ph(0)$.
Since $\Pb(0)$ and $\Ph(0)$ are complementary orthogonal projections, with
\[
\operatorname{rank}{\Pb(0)} = \operatorname{rank}{\Pb(t)} \equiv \operatorname{dim}{\Hc}, 
\quad\mbox{and}\quad
\operatorname{rank}{\Ph(0)} = \operatorname{rank}{\Ph(t)}\equiv \operatorname{dim}{\Kc}, 
\]
we must actually have $\ker{(\Pb(0))} = \Kc$ and  $\ker{(\Ph(0))} = \Hc$.
We also therefore have $\im{\Pb(0)} = \Hc$, and $\im{\Ph(0)} = \Kc$, and it follows that $\Hc$ and $\Kc$ are orthogonal
with respect to $g(0)$. Since $\Pb(0)$ and $\Ph(0)$ are parallel,  $\Hc$ and
$\Kc$ are invariant under $\nabla_{g(0)}$-parallel translation by Lemma \ref{lem:vb}.  Finally, since $(\Rm\circ \Ph)(t) \equiv 0$
for $0 < t \leq T$, it follows that $\left.\Rm(0)\right|_{\Kc}: \Kc\to \WM$ is also the zero map.  
The symmetry of $\Rm$ then implies that $\im{\Rm(0)}\subset \Hc$ and, by Lemma \ref{lem:ambrosesinger},
we conclude that $\hol_p(g(0)) \subset \Hc$, completing the proof.
\end{proof}

\begin{remark} By a result of S. Bando \cite{Bando} (see also Remark 13.32 of \cite{RFV2P2}), 
if $g(t)$ is a complete solution of \eqref{eq:rf} of bounded curvature, then $(M, g(t))$ is a real-analytic manifold
for $0< t \leq T$. Hence at any $t > 0$, any representative $\hol_p(g(t))$ of the isomorphism
class of $\hol(g(t))$
is generated by the set
\[
  \bigcup_{l=0}^{\infty}
    \left\{\;\nabla_{X_1}\nabla_{X_2}\cdots\nabla_{X_l}\Rm(p, t)(\omega)\;\right|\;
      \left. X_1, X_2, \ldots, X_l \in T_pM,\quad \omega\in \WMp\;\right\}. 
\]
(See \cite{KobayashiNomizu}, Sections II.10, III.9.)
Thus we can localize Theorem \ref{thm:rangerm} somewhat: 
If, at some $p\in M$, the endomorphisms coming from the covariant derivatives of $\Rm(g(T))$ of all orders
are contained in some subalgebra $H_p\subset \WMp$, then, at every $q$, $\hol_q(g(T))$
is contained in a subalgebra isomorphic to $H_p$. We can then apply Theorem \ref{thm:rangerm} to conclude
that, for all $(q, t)\in M\times [0, T]$, $\hol_q(g(t))$ is contained a subalgebra isomorphic to $H_p$.
In particular, if $g(T)$ admits a splitting on some neighborhood $U\subset M$ at some time $T >0$, $g(t)$ must
split on a neighborhood of every $p\in M$ at all times $0 \leq t \leq T$.
\end{remark}

\appendix

\section{An alternative proof of the non-expansion of $\Hol^0(g(t))$}\label{app:nonexprooftwo}

In this section we present a second and essentially self-contained proof of Theorem \ref{thm:holonomynonex},
using the general framework of Theorem \ref{thm:holonomy} (but different methods). Although 
we do not use the maximum principle for systems in \cite{Hamilton4D}, the argument is close to 
that suggested by Hamilton for Theorem 4.1 of \cite{HamiltonSingularities}. We include it here
only for reference and comparison purposes.  

Theorem \ref{thm:holonomynonex} has the following infinitesimal reformulation, corresponding
to Theorem \ref{thm:rangerm}. 
\begin{claim*} Suppose $\mathcal{H}\subset\WM$ is a smooth subbundle that is invariant under 
$\nabla_{g(0)}$-parallel transportation 
and the bracket $[\cdot, \cdot]_{g(0)}$. If $\Rm(g(0))\subset \mathcal{H}$, then it follows
that
$\Rm(g(t))\subset \mathcal{H}$ for all $t$ and that $\mathcal{H}$ remains invariant by $\nabla_{g(t)}$-parallel transport
and the bracket $[\cdot, \cdot]_{g(t)}$.  In particular, $\hol_p(g(t))\subset \mathcal{H}_p$.
\end{claim*}

Similar to the proof of Theorem \ref{thm:rangermredux}, we
extend the projection operators $\Pb_0$ and $\Ph_0$ onto $\Hc$ and $\Kc = \Hc^{\perp}$ at time
$t = 0$ to operators $\Pb(t)$ and $\Ph(t)$ for $t > 0$.  The key difference is that we accomplish
this by the solution of a linear parabolic equation rather than by an ODE.
Although with this choice we lose (temporarily) the assurance that the maps remain orthogonal projections, 
it allows us to effectively decouple our system and reduce the number of components
from four to two.

\begin{proof}[Proof of the claim]
Denote by $\mathcal{K} \subset \WM$ the orthogonal complement of $\mathcal{H}$ and by $\Pb_0$, $\Ph_0$ 
the projections onto $\Hc$ and $\Kc$ taken with respect to the metric induced by $g(0)$.

By assumption, $g(t)$ is complete and $\Rm(g(t))$ uniformly bounded, and so we can define $\Pb(t)$ and $\Ph(t)$ on
$M\times [0, T]$ to be the unique bounded solutions to the equations
\begin{align}
  \left(\pdt - \Delta\right)\Pb_{abcd} &= -R_{ap}\Pb_{pbcd} -R_{bp}\Pb_{apcd} - R_{cp}\Pb_{abpd} - R_{dp}\Pb_{abcp}\\
\label{eq:phpar}  \left(\pdt - \Delta\right)\Ph_{abcd} &= -R_{ap}\Ph_{pbcd} -R_{bp}\Ph_{apcd} - R_{cp}\Ph_{abpd} - R_{dp}\Ph_{abcp}
\end{align}
with $\Pb(0) = \Pb_0$ and $\Ph(0) = \Ph_0$. As functions on $\OMt$, the above equations are
\[
  (D_t - \Delta)\Pb_{abcd} = 0, \qquad (D_t - \Delta) \Ph_{abcd} = 0.
\]

Since $\Hc$ is parallel initially,  $A_{mijkl} \defn \nabla_m\Ph_{ijkl} \equiv 0$ initially by
Lemma \ref{lem:vb}.  We claim $A_{mijkl}\equiv 0$ for all $ 0 \leq t \leq T$.
Its evolution is
\begin{align*}
  (D_t - \Delta)A_{mijkl} &= \left[(D_t - \Delta), \nabla_m\right] \Ph_{ijkl}\\
			  &= 2(R_{mbdc}\Lambda_d^c\nabla_b +R_{mc}\nabla_{c})\Ph_{ijkl}, 
\end{align*}
and so
\[
  \left(\pdt - \Delta\right)A = R \ast A,
\]
that is,
\[
    \left|\left(\pdt - \Delta\right)A\right| \leq C|A|
\]
for $C = C(n, K)$.
Defining $Q = |A|^2$, we thus have
\begin{align*}
  \left(\pdt - \Delta\right)Q &= -2|\nabla\nabla\Ph|^2 + 2 \left\langle\left(\pdt - \Delta\right)A, A\right\rangle\\
      &\leq 2C Q,
\end{align*}
so
\[
      Q(x, t) \leq e^{2CT}\sup_{x\in M} Q(x, 0) = 0
\]
on $M\times [0, T]$ by the maximum principle.

Strictly speaking, when $M$ is non-compact, our use of the maximum principle requires some justification.
Since $M$ has bounded curvature (and, in particular, a lower bound on $\Rc(g(t))$), we need only to verify
that $Q$ does not grow too quickly at infinity.  We omit the full details of this verification, 
but point out that, for example, one could use a Bernstein-type trick, as in \cite{Shi}, 
and consider the quantity $F\defn (L + |\Ph|^2)Q$ where $L > 0$ is constant.
Then $F$ satisfies $F(p, 0) \equiv 0$, and, if $L = L(n, \sup|\Ph|^2)$ is sufficiently large, the equation
 \[
    \left(\pdt - \Delta\right) F \leq C_1 F  - C_2 F^2,
\]
for positive constants $C_i = C_i(K_0, L, n)$. Using a standard cutoff function and the maximum principle, one can prove
\[
    \sup_{B_{g(t)}(p, \rho)\times [0, T]}F(x, t)\leq C_3(n, K_0, L, T)\left(\frac{\rho + 1}{\rho}\right)
\]
for all $\rho > > 0$. Hence, upon sending $\rho\to\infty$, one obtains that $Q = |\nabla \Ph|^2 \leq C$ on $M\times [0, T]$.

We conclude, in any case, that $\Ph$ remains parallel, and must actually satisfy
the ODE $D_t \Ph = 0$.  Likewise, we have $\nabla\Pb = 0$ and $D_t \Pb = 0$. But, by Lemmas \ref{lem:vb} and \ref{lem:proj},
this implies that $\Pb$ and $\Ph$ remain complementary
projections, and hence that $H(t) \defn \im{\Pb(t)}$ and $K(t) \defn \im{\Ph(t)}$ remain complementary orthogonal $\nabla_{g(t)}$-parallel subbundles, 
with $H(t)$ invariant under the bracket $[\cdot, \cdot]_{g(t)}$.  In particular
$\mathcal{T}[\Pb, \Ph, \Pb] \equiv 0$  by Lemma \ref{lem:la2}.

Now we define $\Rh = \Rm\circ \Ph$ as before.  We have $\Rh(0) \equiv 0$ by assumption, and claim $\Rh(t) \equiv 0$ for
all $0\leq t \leq T$.
Since $\nabla \Ph \equiv 0$,
\[
  (D_t - \Delta) \Rh = \Qc(\Rm)\circ\Ph.
\]
Using $\mathcal{T}[\Pb, \Ph, \Pb]\equiv 0$, we have, by \eqref{eq:qcomp} and Shi's estimates,
\[
  \left|\Qc(\Rm)\circ \Ph \right|^2 \leq C|\Rh|^2.
\]
  So the (uniformly bounded) quantity $W = |\Rh|^2$ satisfies
\[
  \left(\pdt - \Delta\right) W \leq C W
\]
with $W(0)\equiv 0$; thus $W(t) \equiv 0$ by the maximum principle.  Hence $\im{\Rm(t)} \subset H(t)$.
Applying Proposition \ref{lem:imagerm} shows that $H(t) \equiv \mathcal{H}$ and $K(t) \equiv \mathcal{K}$,
and the theorem is proved.      
\end{proof}

\begin{acknowledgement*}
    The author wishes to thank Professors Bennett Chow, Gerhard Huisken, and Lei Ni for
    their support and encouragement.
\end{acknowledgement*}

\end{document}